\documentclass[a4paper,11pt,leqno,english]{smfart}
\usepackage{aeguill}
\usepackage{enumerate}
\usepackage{amssymb,amsmath,latexsym,amsthm}
\usepackage[T1]{fontenc}
\usepackage{smfthm}      
\usepackage{geometry}   
\usepackage{url}      
\usepackage[frenchb, english]{babel}
\usepackage[utf8]{inputenc}   
\usepackage{mathrsfs}
\usepackage{relsize}
\usepackage{xcolor}
\usepackage{comment}
\definecolor{violet}{rgb}{0.0,0.2,0.7}
\definecolor{rouge2}{rgb}{0.8,0.0,0.2}
\usepackage{hyperref}
\usepackage{mathrsfs}
\newcommand{\R}{\mathbb{R}}

\newcommand{\Q}{\mathbb{Q}}

\renewcommand{\O}{\mathcal{O}}
\newcommand{\C}{\mathcal{C}}

\newcommand{\ep}{\varepsilon}

\newcommand{\la}{\langle}

\newcommand{\ra}{\rangle}

\renewcommand{\geq}{\geqslant}
\renewcommand{\leq}{\leqslant}
\newcommand{\Ric}{\mathrm{Ric} \,}

\newcommand{\Supp}{\mathrm {Supp}}
\newcommand{\ord}{\mathrm {ord}}
\newcommand{\rk}{\mathrm {rk}}

\newcommand{\tr}{\mathrm{tr}}

\newcommand{\TB}{\otimes ^mT^\star_X\langle B\rangle}

\newcommand{\F}{\mathscr{F}}

\newcommand{\ddc}{dd^c}

\newcommand\cF{{\mathcal{F}}}
\newcommand\cO{{\mathcal{O}}}
\newcommand\cI{{\mathcal{I}}}

\def\Vol{\mathop{\rm Vol}\nolimits}

\def\dbar{\overline\partial}

\let\ol=\overline

\let\wh=\widehat

\newtheorem{cor}{Corollary}
\newtheorem{lem}{Lemma}
\newtheorem{defn}{Definition}

\numberwithin{equation}{section}

\begin{document}

\frontmatter 

\title[Holomorphic tensors]{Positivity properties of the bundle of logarithmic tensors 
on compact K\"ahler manifolds}

\date{\today}
\author{Fr\'ed\'eric Campana}
\address{Institut Elie Cartan, Nancy \\
Universit\'e de Lorraine \\
Nancy}
\email{Frederic.Campana@univ-lorraine.fr}

\author{Mihai P\u{a}un}
\address{Korea Institute for Advanced Study\\
School of Mathematics, 85 Hoegiro, Dongdaemun-gu, Seoul 130-722, Korea}
\email{paun@kias.re.kr}

%And this dramatically good paper was written by us, MP & FC.


\begin{abstract} 
Let $X$ be a compact K\"ahler manifold, endowed with an effective reduced
divisor $B= \sum Y_k$ having simple normal crossing support. We consider a closed form
of (1,1)-type $\alpha$ on $X$ which is semi-positive, such that the class $c_1(K_X+ B)+ \{\alpha\}\in H^{1,1}(X, \R)$ is pseudo-effective. A particular case of the first result we establish in this short note states the following. Let $m$ be a positive integer, and let $L$ be a line bundle 
on $X$, such that there exists a generically injective morphism 
$L\to \otimes ^mT_X^\star\langle B\rangle$, where we denote by $T_X^\star\langle B\rangle$ 
the logarithmic co-tangent bundle associated to the pair $(X, B)$. Then for any
K\"ahler class $\{\omega\}$ on $X$, we have the inequality
$$\displaystyle \int_Xc_1(L)\wedge \{\omega\}^{n-1}\leq m
\int_X(c_1(K_X+ B)+ \{\alpha\})\wedge \{\omega\}^{n-1}.$$ If $X$ is projective, then
this result gives a generalization of a criteria due to Y.~Miyaoka, concerning the generic semi-positivity: under the hypothesis above, let $Q$ be the quotient of $\otimes ^mT_X^\star\langle B\rangle$ by $L$. Then its degree on a generic complete intersection curve $C\subset X$ is bounded from below by 
$$\displaystyle \Big(\frac{n^m-1}{n-1}- m\Big)\int_C(c_1(K_X+ B)+ \{\alpha\})-\frac{n^m-1}{n-1}\int_C\alpha.$$  
As a consequence, we obtain a new proof of one of the main results of our preceding work 
\cite{CP}. This proof is largely inspired by the theory of orbifolds as developped by the first named author, although we are not
making an explicit use of this theory.
\end{abstract}

\maketitle

\tableofcontents

\section{Introduction}

Let $X$ be a compact K\"ahler manifold of dimension $n$. Let 
$B= \sum_{j\in J}Y_j$ be a reduced ``boundary" divisor on $X$, whose support has simple normal crossings. 

We recall that the logarithmic co-tangent bundle $T^\star_X\langle B\rangle$ associated to $(X, B)$ 
is defined as follows. Let $U\subset X$ be a coordinate open set, and let 
$(z^1,\dots z^n)$ be a coordinate system of $X$ defined on $U$, such that 
say $B\cap U= (z^1z^2\dots z^p= 0)$. Such a coordinate system will be called \emph{adapted} to the pair
$(X, B)$. When restricted to $U$, the bundle $T^\star_X\langle B\rangle$ corresponds to the locally free sheaf generated by
$$\frac{d z^1}{z^1}, \dots, \frac{d z^p}{z^p}, dz^{p+1},\dots dz^n.$$

The main topic we will explore in this paper can be formulated as follows. Let $m$ be a positive 
integer, and let $\cF$ be a sub-sheaf of the bundle
$\otimes ^mT_X^\star\langle B\rangle$. Let $\omega$ be a K\"ahler class on $X$; we would like to 
obtain a upper bound of the intersection number
\begin{equation}
\int_Xc_1(\cF)\wedge \omega^{n-1}
\end{equation}
in terms of the first Chern class of the bundle $T_X^\star\langle B\rangle$.
\smallskip

We observe that this kind of questions has a meaning in ``abstract" context. Indeed, we 
consider the following data: let $(E, h_E)$ be a holomorphic Hermitian vector bundle, and let 
$(L, h_L)$ be a Hermitian line bundle. We assume that  
$H^0(X, E\otimes L^{-1})\neq 0$, and let $u$ be a non-identically zero holomorphic section of
$E\otimes L^{-1}$. 

We denote by $|u|^2$ the pointwise norm of $u$, measured with respect to the metrics we have 
on $E$ and $L$, respectively. Then the analogue of the Poincar\'e-Lelong formula (cf. e.g. \cite{D1})
gives
\begin{equation}\ddc\log |u|^2\geq \Theta_{h_L}(L)- \frac{\langle\Theta_{h_E}(E)u, u\rangle}{|u|^2}.\end{equation}
On the other hand, the quantity we are interested in 
is equal to $\displaystyle \int_X \Theta_{h_L}(L)\wedge \omega^{n-1}$; by the inequality (1.2) combined with Stokes formula we obtain
$$0\geq \int_X \Theta_{h_L}(L)\wedge \omega^{n-1}- \int_X \frac{\langle\Theta_{h_E}(E)u, u\rangle}{|u|^2}\wedge \omega^{n-1}.$$
In conclusion, the degree of $L$ with respect to $\omega$ is bounded from above as follows
\begin{equation}\int_X \Theta_{h_L}(L)\wedge \omega^{n-1}\leq \int_X \frac{\langle\Theta_{h_E}(E)u, u\rangle}{|u|^2}\wedge \omega^{n-1}.\end{equation}
We denote by $\Lambda_\omega$ the contraction operator corresponding to $\omega$; his action on the curvature tensor is given by the 
formula
$$\frac{\Theta_{h_E}(E)\wedge \omega^{n-1}}{\omega^n}= \Lambda_\omega \Theta_{h_E}(E)$$
so that $\Lambda_\omega \Theta_{h_E}(E)$ is an endomorphism of $E$, called the 
\emph{mean curvature} of $E$.

If the bundle $E$ is stable with respect to $\omega$, then it is known cf. \cite{Don85} that we can construct a
Hermite-Einstein metric $h_E$ so that the endomorphism $\Lambda_\omega \Theta_{h_E}(E)$ is a multiple of the 
identity. In this case, the expression in the right hand side of (3) is quickly computed in terms of the $\omega$-degree of $E$. But in general, the mean curvature term seems to be very difficult to control.
\smallskip

\noindent 
In the article \cite{Enoki}, I. Enoki found a very elegant way of dealing with this question. He observes that 
if the bundle $E$ i.e. equal e.g. to the cotangent bundle of $X$, then the mean curvature 
of $E$ coincides with $-\Ric_\omega$, i.e. the curvature of the canonical bundle. In conclusion, via the Monge-Amp\`ere equation, it 
would be enough to assume that the canonical bundle of $X$ has a (weak) positivity property in order to derive  
a bound for the degree of $L$ with respect to $\omega$. 
\medskip

 The approach sketched above, combined with some recent results in the framework of metrics with 
conic singularities allows us to establish the following statement.

\begin{theo} Let $m$ be a positive 
integer, and let $\cF$ be a torsion free sub-sheaf of the bundle
$\otimes ^mT_X^\star\langle B\rangle$. We consider a semi-positive definite, closed (1,1)-form $\alpha$ on $X$, such that the class 
$$c_1(K_X+ B)+ \{\alpha\}$$
is pseudo-effective. Let  
$\omega_0$ be a K\"ahler class on $X$; then we have
\begin{equation} \frac{1}{\rk(\F)}\int_Xc_1(\F)\wedge \{\omega_0\}^{n-1}\leq m
\int_X(c_1(K_X+ B)+ \{\alpha\})\wedge \{\omega_0\}^{n-1}.
\end{equation}
\end{theo}

\noindent As a consequence, we obtain a new proof of the following result. Our arguments here are largely inspired by 
the original ones, cf. \cite{CP}.

\begin{theo} \cite{CP} Let $X$ be a projective manifold;
we assume that there exist a positive integer $m$ and a big line bundle $L$ such that 
\begin{equation}H^0(X, \otimes ^m T^\star_X\langle B\rangle\otimes L^{-1})\neq 0.
\end{equation}
Then $K_X+ B$ is big.
\end{theo}

\medskip

 In the first part of our note we will briefly recall the notion of pseudo-effective classes 
in K\"ahler context,
as well as a few results concerning the Monge-Amp\`ere equations. The later topic concerns
the so-called \emph{metrics with conic singularities}; specifically, the result we need was established in \cite{CGP}. We remark here that the uniformity of the estimates obtained in \cite{CGP}
will allow us a ``smooth" presentation of the proof. 

Roughly speaking, the proof of Theorem 1.1 follows by using the approach initiated by I. Enoki mentioned above, combined with the estimates we have at our disposal in the Monge-Amp\`ere theory. 
As for Theorem 1.2, we basically follow the ``second proof" in \cite{CP}, and use Theorem 1.1 
as main technical support, instead of the generic semi-positivity of orbifold cotangent bundle
in the original argument.

Parts of the techniques we will use here are equally employed by H. Guenancia in \cite{G2}, in order to
establish the stability of the tangent bundle of the canonically polarized manifolds in a singular setting. 
Finally, we mention here that 
in order to obtain a stronger version of Theorem 1.1,
the best one can expect would be to replace $\omega_0^{n-1}$ with a mobile class. Our hope is that the methods we develop here might be useful in this context.

\medskip

\noindent
\textbf{Acknowledgements.}
It is our privilege to thank H. Guenancia for many discussions concerning the topics analyzed here, and for kindly pointing out a 
rather severe flaw in a previous version of this paper. Part of the present article was prepared during the conference \emph{Complex Monge-Amp\`ere equations}, hosted by BIRS in 2014. We would like 
to thank the organizers P. Eyssidieux, V. Guedj and S. Boucksom for the invitation, and to express our gratitude for the excellent working conditions provided by BIRS.

\section{Pseudo-effective classes and Monge-Amp\`ere equations}

Let $X$ be a compact complex manifold, and let $\{\gamma\}\in H^{1,1}(X, \R)$ be a real cohomology class of type (1,1). By convention, we assume that $\gamma$ is a closed, non-singular (1,1)
form on $X$; we recall the following definition.

\begin{defn}
We say that the class $\{\gamma\}$ is pseudo-effective if it contains a closed positive current. 
This is equivalent to the existence of a function $f\in L^1(X)$ such that 
\begin{equation}
T:= \gamma + \ddc f\geq 0.  
\end{equation}
\end{defn}

Let $\{\gamma\}$ be a pseudo-effective class.
If some multiple of it equals the first Chern class of 
a $\Q$-line bundle $L$ and if $X$ is projective, 
then the bundle $L$ is pseudo-effective in the sense of algebraic geometry,
i.e. given $p$ a positive integer and an ample line bundle $A$, there exists a constant $C= C(p, A)> 0$ such that
$$h^0\big(X, k(pL+A)\big)\geq Ck^n$$
as $k\to \infty$.
In  other words, the Kodaira dimension of $pL+ A$ is maximal. This statement is a consequence of \cite{D2}, and can be seen as a generalization of the Kodaira embedding theorem. Even if the manifold $X$ is projective, the property of pseudo-effective line bundles (in the sense of algebraic geometry) revealed by the previous definition is important: such a line bundle may have negative 
Kodaira dimension, but nevertheless, its first Chern class carries a closed positive current.

While dealing with pseudo-effective classes on compact K\"ahler manifolds, one seldom works directly with the current $T$ above, mainly because it may be too singular. The regularization theorem we will state next is therefore very useful. Before that, we recall the following notion.

\begin{defn}
A function $\psi$ on $X$ has logarithmic poles if for each coordinate set $U\subset X$ there exists a family of holomorphic functions $\sigma_j\in \O(U)$ such that we have
\begin{equation}
\psi\equiv c\log \big(\sum_j |\sigma_j|^2\big) \hbox { mod } \C^\infty(U)
\end{equation} 
where $c$ is a positive constant. In other words, the function $\psi$ is locally equivalent with
the logarithm of a sum of squares of absolute values of holomorphic functions.
\end{defn}
\medskip

\noindent The following regularization result will play an important role here.

\begin{theo} \cite{D2} \label{reg}
Let $T= \gamma + \ddc f\geq 0$ be a closed positive current. Then there exists a sequence $(f_\eta)_{\eta> 0}$ of functions with logarithmic poles, such that
$$T_\eta:= \gamma + \ddc f_\eta\geq -\eta\omega$$ 
for all $\eta$, and such that $f_\eta\to f$ in $L^1(X)$.
\end{theo}
\medskip

\noindent We recall next a few results in the theory of the Monge-Amp\`ere equations. 
For the rest of our paper, the manifold $X$ will be assumed to be compact
K\"ahler, and let $\omega_0$ be a fixed, reference metric on $X$. 
	
Theorem 1.1 states that a certain numerical positivity property 
of the bundle $\otimes ^mT^\star_X\langle B\rangle$ holds, provided that the negative part of the Chern class of $K_X+ B$ is bounded. 

The requirement concerning the bundle $K_X+ B$ in Theorem 1.1 is equivalent to the existence of a singular 
volume element on $X$, whose associated curvature current is greater than $-\alpha$.
The link between this kind of positivity (or negativity) properties of the canonical bundle 
and the differential geometry of $X$ is given by the 
Monge-Amp\`ere theory, more precisely, by the Aubin-Calabi-Yau theorem. In its original formulation 
\cite{Yau78}, this result states that given a non-singular volume element on $X$, whose total mass is
equal to $\displaystyle \int_X\omega_0^n$, there exists a K\"ahler metric 
$\omega\in \{\omega_0\}$ whose determinant is equal to the said volume element.
\medskip

\noindent Partly motivated by questions arising from algebraic geometry 
many results generalizing the Calabi-Yau theorem for singular volume elements were established recently. The following statement will be the main technical tool in our proof of Theorem 1.1.

\begin{theo} \label{MA} \cite{CGP} Let $(X, \omega_0)$ be a compact K\"ahler manifold, and let 
$$B= \sum_jY_j$$ 
be a reduced divisor, whose support has simple normal crossings. We consider
a finite set of functions $(\psi_r)_{r=1,\dots A}\subset L^1(X)$ with logarithmic poles, normalized such that $\displaystyle \int_X\psi_r dV_0= 0$
for each $r$, where $dV_0$ is the volume element on $X$ induced by the metric $\omega_0$. Let $C_\psi$ be a positive constant, such that 
$$C_\psi\omega_0 + \ddc\psi_r\geq 0.$$
Then for each set of parameters $\Lambda= (\lambda, \varepsilon, \delta)\in ]0, 1]^{A}\times [0,1]\times ]0, 1/2]$  
there exists a positive (normalization) constant $C_\Lambda$ and a K\"ahler metric
$\displaystyle \omega_{\Lambda}= \omega_0+ \ddc\varphi_{\Lambda}$, such that 
$$\omega_{\Lambda}^n= C_\Lambda\frac{\prod_r\big(\lambda^2_r+ \exp(\psi_r)\big)}
{\prod_j(\varepsilon^2+ |s_j|^2)^{1-\delta}}\omega_0^n,\leqno {\rm (MA)}$$
and such that the following uniform estimates hold.

\begin{enumerate}
\item[{\rm (i)}] There exists a constant $C_\delta> 0$ depending on $\delta, C_\psi$ and $(X, \omega_0)$
but uniform with respect to $\lambda, \varepsilon\in [0, 1/2]^{A+1}$ such that $\sup_X|\varphi_{\Lambda}|\leq C_\delta$.
\smallskip

\item[{\rm (ii)}] For each coordinate system $(z^1,\dots, z^n)$ adapted to $(X, B)$ on $U$
we have 
$$\omega_{\Lambda}\leq C_\delta \Big(\sum_{j=1}^p\frac{\sqrt -1dz^j\wedge dz^{\ol j}}{(\varepsilon^2+ |z^j|^2)^{1-\delta}}+ \sum_{j=p+1}^n\sqrt -1dz^j\wedge dz^{\ol j}\Big)$$
in other words, the solution $\omega_\Lambda$ is bounded from above by the standard conic metric on each coordinate set. 
\smallskip

\item[{\rm (iii)}] There exists a constant $C_{\lambda, \delta}$ uniform with respect to $\varepsilon$ such that
$$\omega_{\Lambda}\geq C_{\lambda, \delta} \Big(\sum_{j=1}^p\frac{\sqrt -1dz^j\wedge dz^{\ol j}}{(\varepsilon^2+ |z^j|^2)^{1-\delta}}+ \sum_{j=p+1}^n\sqrt -1dz^j\wedge dz^{\ol j}\Big).$$

\end{enumerate}
\end{theo}

\noindent In the statement above we denote by $s_j$ the section of the bundle $\O(Y_j)$ whose zero set
equals $Y_j$, and we denote by $|s_j|$ the (pointwise) norm of $s_j$ measured with a 
non-singular metric $h_j$ on $\O(Y_j)$. We denote by $\theta_j$ the curvature of the 
bundle $\O(Y_j)$ with respect to $h_j$.

\begin{rema}  We see that (unlike in all the other recent articles concerning the metrics with conic singularities) in the statement above we have obtained uniform 
the estimates with respect to the parameter $\varepsilon\geq 0$. This will be
play an important role in the next section.
\end{rema} 

\medskip

\noindent In the remaining part of this section, we will evaluate the curvature of the canonical bundle $K_X$ endowed with the determinant of the metric given by the equation (MA). The formula reads as
\begin{equation}
\Theta_{\omega_\Lambda}(K_X)= \Theta(K_X)+ \sum_r\ddc \log\big(\lambda^2+ \exp(\psi_r)\big)
-(1-\delta)\sum_j\ddc \log (\varepsilon^2+ |s_j|^2).
\end{equation}
where $\Theta(K_X):= \Theta_{\omega_0}(K_X)$ is the curvature of the canonical bundle with respect to the volume element induced by the reference metric
$\omega_0$.
We expand next the Hessian terms in (2.3), as follows

\begin{eqnarray}
\ddc \log\big(\lambda^2+ \exp(\psi)\big)&= & \frac{\exp(\psi)}{\lambda^2+ \exp(\psi)}\ddc \psi+
\frac{\lambda^2\exp(\psi)}{\big(\lambda^2+ \exp(\psi)\big)^2}\sqrt{-1}\partial\psi\wedge\dbar \psi 
\label{Hess1} \nonumber \\
&= & \ddc\psi + \frac{\lambda^2\exp(\psi)}{\big(\lambda^2+ \exp(\psi)\big)^2}\sqrt{-1}\partial\psi\wedge\dbar \psi- \frac{\lambda^2}{\lambda^2+ \exp(\psi)}\ddc \psi \nonumber
\end{eqnarray}
and 

\begin{eqnarray}
\ddc \log(\varepsilon^2+ |s_j|^2)&= & \frac{\ep^2}{(\ep^2+ |s_j|^2)^2}\sqrt{-1}\partial s_j
\wedge \ol{\partial s_j} - \frac{|s_j|^2}{\ep^2+ |s_j|^2}\theta_j \nonumber \\
&= & -\theta_j+ \frac{\ep^2}{(\ep^2+ |s_j|^2)^2}\sqrt{-1}\partial s_j
\wedge \ol{\partial s_j} +  \frac{\varepsilon^2}{\ep^2+ |s_j|^2}\theta_j, \nonumber
\end{eqnarray}
respectively. Rearranging the terms, we obtain the following identity

\begin{eqnarray}
\Theta_{\omega_\Lambda}(K_X) &+ & (1-\delta)\sum_j\frac{\ep^2}{(\ep^2+ |s_j|^2)^2}\sqrt{-1}\partial s_j
\wedge \ol{\partial s_j} + \delta\sum_j\theta_j+ \nonumber  \\
&+ & \sum_r\frac{\lambda^2}{\lambda^2+ \exp(\psi_r)}\ddc \psi_r
- (1-\delta)\sum_j\frac{\varepsilon^2}{\ep^2+ |s_j|^2}\theta_j\nonumber  \\
&= &  \Theta(K_X)+ \sum_j\theta_j+  \sum_r\ddc\psi_r + \sum_r\frac{\lambda^2\exp(\psi_r)}{\big(\lambda^2+ \exp(\psi_r)\big)^2}\sqrt{-1}\partial\psi_r\wedge\dbar \psi_r. \nonumber
\end{eqnarray}

The relationship between this expression and Theorem 1.1 is easy to understand: if we add
the semi-positive form $\alpha$, the right hand side term can be assumed to be positive
by choosing $\psi$ in an appropriate manner (according to the pseudo-effectivity hypothesis). Thus, the curvature term
$\displaystyle \Theta_{\omega_\Lambda}(K_X)$ is naturally written as the \emph{difference of two semi-positive forms}, modulo some undesirable ``error" terms like
\begin{equation}
 \sum_r\frac{\lambda^2}{\lambda^2+ \exp(\psi_r)}\ddc \psi_r+ \sum_j\Big(\delta- \frac{(1-\delta)\varepsilon^2}{\ep^2+ |s_j|^2}\Big)\theta_j
\end{equation}
which will converge to zero, provided that we are 
choosing the set of parameters $\Lambda$ properly. It is at this point that the estimates in
Theorem \ref{MA} are playing a determinant r\^ole.

%%%%%%%%%%%%%%%%%%%%%%%%%%%%%%%%%%%%%%%%%%%%%%%%%%%%%%%%%%%%%%%%%%%%%%%%%%%%%%%%%%%%%%%%%%%%%%%%
\section{Proof of Theorem 1.1}

In this section we will unfold our arguments for the proof of Theorem 1.1; we proceed in several steps. A first standard observation is that it is enough to prove this statement in case $\cF:= L$ a line bundle 
on $X$. The reduction to this case is classical, as follows. If we denote by $p$ the rank of 
$\cF$ at the generic point of $X$, then we have a sheaf injection 
$$0\to \Lambda^p\cF\to \otimes ^{mp}T^\star_X\langle B\rangle$$
and our claim follows by passing to the bi-dual; we refer to \cite{Koba} for further details.
\medskip

\noindent 
\textbf{Step 1.} By hypothesis, the class $c_1(K_X+ B)+\{\alpha \}$ is pseudo-effective. When combined with Theorem \ref{reg}, we infer the existence of a sequence of functions $(f_\eta)_{\eta> 0}$ such that 
the following holds:

\begin{enumerate}
\smallskip

\item[(1)] For each $\eta> 0$, the function $f_\eta$ has logarithmic poles, and $\int_Xf_\eta dV_0= 0$.
\smallskip

\item[(2)]  We have $\Theta(K_X)+ \sum_j\theta_j+ \alpha+ \ddc f_\eta\geq -\eta\omega_0$
in the sense of currents on $X$ (we are using here the same notations as in the preceding section).
\end{enumerate}

\smallskip

\noindent The relation (2) above is equivalent to the positivity of the current
$$T_\eta= \Theta(K_X)+ \sum_j\theta_j+ \alpha+ \ddc f_\eta+ \eta\omega_0;$$
for some technical reasons which will appear later on in the proof, we have to ``separate"
the codimension 1 singularities of $T_\eta$ from the rest. For a general closed positive current
this is indeed possible thanks to a result due to Y.-T. Siu in \cite{Siu}. In the present situation, 
this is much simpler, since $f_\eta$ has logarithmic poles; in any case, we write
\begin{equation}
T_\eta= \sum_{r=1}^{A_\eta}b^r_\eta [W_{r, \eta}]+R
\end{equation}
where the singularities of the closed positive current $R$ are in codimension 2 or higher. 
Let $g_{r,\eta}$ be an arbitrary, non-singular metric on the line bundle $\cO(W_{r, \eta})$ associated to 
the hypersurface $W_{r, \eta}$, and let $\beta_{r, \eta}$ be the associated curvature form. We denote by
$\sigma_{r, \eta}$ the tautological section of $\cO(W_{r, \eta})$, whose set of zeroes is
precisely $W_{r, \eta}$, and we define
\begin{equation}
\wh f_\eta:= f_\eta- \sum_{r=1}^{A_\eta}b^r_\eta\log |\sigma_{r, \eta}|^2.
\end{equation}
Since the current $R$ is positive we have
\begin{equation}\label{pos}
\Theta(K_X)+ \sum_{j=1}^N\theta_j+ \alpha-\sum_{r=1}^{A_\eta}b^r_\eta\beta_{r, \eta}+ \ddc \wh f_\eta+ \eta\omega_0\geq 0
\end{equation}
in the sense of currents on $X$.

\noindent We apply next Theorem \ref{MA}: for each set of parameters 
$$\Lambda:= (\varepsilon, \lambda, \eta, \rho, \delta)\in ]0, 1]\times ]0, 1]\times ]0, 1]\times ]0, 1]\times ]0, 1/2]$$ there exists a K\"ahler metric 
$$\omega_\Lambda:= \omega_0+ \ddc \varphi_\Lambda,\quad \int_X\varphi_\Lambda dV_0= 0$$
such that
\begin{equation}
\omega_{\Lambda}^n= C_\Lambda\frac{\big(\lambda^2+ \exp(\wh f_\eta)\big)\prod_r(\rho^2+ |\sigma_{r, \eta}|^2)^{b^r_\eta}}
{\prod_j(\varepsilon^2+ |s_j|^2)^{1-\delta}}\omega_0^n
\end{equation}
where $C_\Lambda$ is a normalization constant. By estimate (ii) of Theorem \ref{MA}, we have
\begin{equation}\label{estma}
\frac{1}{C_{\delta, \eta}}\omega_{\Lambda}\leq \sum_{j=1}^p\frac{\sqrt -1dz^j\wedge dz^{\ol j}}{(\varepsilon^2+ |z^j|^2)^{1-\delta}}+ \sum_{j=p+1}^n\sqrt -1dz^j\wedge dz^{\ol j}
\end{equation}
on a small coordinate set, where $(z^i)$ are local coordinates adapted to $(X, B)$. 

\begin{rema} Actually, the constant $C_{\delta, \eta}$ as above can be assumed to be
independent of $\eta$, provided that $\lambda$ and 
$\rho$ are small enough (depending on $\eta$). This amounts to take
$\rho\to 0$ and $\lambda\to 0$ before taking $\eta\to 0$ in the limiting process at the end of the proof.
However, the estimate as stated in \ref{estma} will be sufficient for us.
\end{rema}

Also, by the computations performed at the end of the preceding section, the curvature of $K_X$
with respect to the metric $\omega_{\Lambda}$ can be expressed as follows

\begin{eqnarray} \label{curb}
\Theta_{\omega_\Lambda}(K_X) &+ & (1-\delta)\sum_{j=1}^N
\frac{\ep^2}{(\ep^2+ |s_j|^2)^2}\sqrt{-1}\partial s_j
\wedge \ol{\partial s_j} + \alpha  + \eta\omega_0 \nonumber \\
&+ & \frac{\lambda^2}{\lambda^2+ \exp(\wh f_\eta)}\ddc \wh f_\eta
+ \sum_{j=1}^N
\Big(\delta- \frac{(1-\delta)\varepsilon^2}{\ep^2+ |s_j|^2}\Big)\theta_j+ 
\sum_{r=1}^{A_\eta} \frac{b^r_\eta\rho^2}{\rho^2+ |\sigma_{r, \eta}|^2}\beta_{r, \eta}
\nonumber \\
&= &  \Theta(K_X)+ \sum_{j=1}^N\theta_j- \sum_{r=1}^{A_\eta}b^r_\eta\beta_{r, \eta}+ 
\ddc \wh f_\eta + \alpha + \eta\omega_0
  \\
&+ & \sum_{r=1}^{A_\eta} \frac{b^r_\eta \rho^2}{(\rho^2+ |\sigma_{r, \eta}|^2)^2}\sqrt{-1}\partial \sigma_{r, \eta}
\wedge \ol{\partial \sigma_{r, \eta}}+ \frac{\lambda^2\exp(\wh f_\eta)}{\big(\lambda^2+ \exp(\wh f_\eta)\big)^2}\sqrt{-1}
\partial \wh f_\eta\wedge\dbar \wh f_\eta. \nonumber
\end{eqnarray}

\noindent In the relation above, we remark that the right hand side is 
positive definite.

\medskip

\noindent 
\textbf{Step 2.} As we have seen at the beginning of this section, the sheaf $\cF$ can be assumed to be a line bundle, say $L$. Let 
$$u\in H^0(X, \TB\otimes L^{-1})$$
be the holomorphic section corresponding to the injection $\displaystyle L\mapsto \TB$ whose existence is insured by hypothesis. We consider an arbitrary, non-singular metric $h_L$ on $L$.

Let $M:= X\setminus \Supp(B)$, and let $\Omega \subset M$ be an open set, such that 
$$\ol\Omega\subset M.$$
The restriction bundle $\displaystyle \TB|_M$ identifies with the tensor power of the 
usual cotangent bundle 
$\displaystyle \otimes^mT^\star_X$, and we can endow it with the metric induced by 
$\omega_\Lambda$. 
Another way of presenting this is that we endow the bundle $\TB$ with a singular metric 
given by $\omega_\Lambda$. Let $\mu$ be a positiv real number;
we have the following inequality
\begin{equation}\label{ddc}
\ddc\log\big(\mu^2+ |s|^{2m\delta}|u|_\Lambda^2\big)
\geq \frac{|s|^{2m\delta}|u|_\Lambda^2}{\mu^2+ |s|^{2m\delta}|u|_\Lambda^2}\big(\Theta_{h_L}(L)- \delta m\Theta_B\big)-
|s|^{2m\delta}\frac{\langle\Theta_{\Lambda}(E_m)u, u\rangle}{\mu^2+ |s|^{2m\delta}|u|_\Lambda^2},\end{equation}
where the notations are as follows: $E_m$ stands for $\TB$, and the symbol $\Theta_{\Lambda}(E_m)$ denotes the curvature of
$E_m$ with respect to the metric induced by $\omega_\Lambda$.
The quantity $|u|_\Lambda$ represents the norm of $u$ measured with respect to the metric induced by $h_L$ and 
$\omega_\Lambda$. We denote by $s$ the canonical section corresponding to the boundary divisor $B$, and 
by $|s|$ its norm, measured with respect to a non-singular metric; the associated curvature form is $\Theta_B$.   
We remark that 
this inequality (valid also in case $\mu=0$)
is nothing but a version of the Poincar\'e-Lelong formula for vector bundles.
On the other hand, we stress that the inequality
\ref{ddc} only holds on sets $\Omega\subseteq M$, since the quantity 
$\displaystyle |s|^{2m\delta}|u|_\Lambda^2$ becomes meromorphic across the support of $B$.

Let $\xi$ be a cut-off function, which equals 1 in the complement of an open set containing 
$\Supp(B)$ and which vanishes in a smaller open set containing 
$\Supp(B)$. We multiply \ref{ddc} with $\xi\omega_\Lambda^{n-1}$ and we integrate over 
$$X_0:= X\setminus \Supp(B);$$ 
we have
\begin{eqnarray}\label{ddcm}
\int_{X_0}\xi \ddc\log\big(\mu^2+ |s|^{2m\delta}|u|_\Lambda^2\big)\wedge\omega_\Lambda^{n-1}& \!\! + &\!\! {\int_{X_0}\xi\frac{|s|^{2m\delta}}{\mu^2+ |s|^{2m\delta}|u|_\Lambda^2}\langle\Theta_{\Lambda}(E_m)u, u\rangle\wedge\omega_\Lambda^{n-1} }\nonumber \\ 
& \geq 
& \!\!
\int_{X_0}\xi\frac{|s|^{2m\delta}|u|_\Lambda^2}{\mu^2+ |s|^{2m\delta}|u|_\Lambda^2}\big(\Theta_{h_L}(L)- m\delta\Theta_B\big)\wedge\omega_\Lambda^{n-1}.
\end{eqnarray}
\medskip

\noindent 
\textbf{Step 3.}
In the following lines we will evaluate the curvature term $\displaystyle {\langle\Theta_{\Lambda}(E_m)u, u\rangle}\wedge\omega_\Lambda^{n-1}$
in \ref{ddcm}. To this end, we consider a point $x_0\in \Omega$, and we will do a pointwise computation by using geodesic coordinates at $x_0$ (here we ignore completely the log structure induced by $B$). 

For any set of parameters $\Lambda$ the metric $\omega_\Lambda$ is K\"ahler, so there exists
a geodesic coordinate system $(z_1, \dots, z_n)$ centered at $x_0$, i.e. near $x_0$ we write 
\begin {equation}
\omega_\Lambda= \sum_{q=1}^n\sqrt{-1}dz_q\wedge d\ol z_{q}+ \sum_{j, k, \alpha, \beta} c^k_{j\alpha\ol \beta}z_j\ol z_{k}
\sqrt{-1}dz_\alpha\wedge d\ol z_{\beta}+ \cO(|z|^3).
\end{equation}
In the expression above, the complex numbers $(c^k_{j\alpha\ol \beta})$ are defined as follows
$$\Theta_\Lambda(T^\star_X)_{x_0}= 
\sum_{j, \alpha, \beta}c^k_{j\alpha\ol \beta}dz_\alpha\wedge d\ol z_{\beta}\otimes dz_j\otimes
\frac{\partial}{\partial z_k},$$
that is to say, they are the coefficients of the curvature tensor of $(T_X^\star, \omega_\Lambda)$.
\medskip

Let 
$$\displaystyle u= \sum_{J= (j_1,\dots, j_m)}u_J dz_{j_1}\!\!\otimes\dots\otimes dz_{j_m}$$
be the local expression of the section $u$. By definition, the curvature of $\TB$ acts on $u$ as follows

\begin{eqnarray}
\Theta_{\Lambda}(E_m)u &= & \sum_{r, J= (j_1,\dots, j_m)}u_J dz_{j_1}\!\!\otimes\dots
\otimes \Theta_\Lambda(T^\star_X)dz_{j_r}\!\!\otimes \dots \otimes dz_{j_m} \nonumber \\
&=& \sum_{J= (j_1,\dots, j_m)}\sum_{r, p, \alpha,\beta}u_Jc^r_{p\alpha\ol \beta} dz_{j_1}\!\!\otimes\dots
\otimes dz_{p}\!\otimes \dots \otimes dz_{j_m}\otimes dz_\alpha\wedge d\ol z_{\beta}. \nonumber 
\end{eqnarray} 

Therefore at the point $x_0$ we have 
\begin{equation}\label{comp}
\frac{\langle\Theta_{\Lambda}(E_m)u, u\rangle\wedge\omega_\Lambda^{n-1}}
{\omega_\Lambda^{n}}= \frac{1}{n}\sum_{J, p, d} u_J\ol u_{J_{\wh p d}}\theta^p_d
\end{equation}
where for any index $J=(j_1,\dots, j_m)$ we denote by $J_{\wh p d}$ the index obtained by replacing 
$j_p$ with $d$ in the expression of $J$ (the other elements are unchanged), and
$$\theta^p_d:= \sum_\alpha c^p_{d\alpha\ol \alpha} $$
are the coefficients of the Hermitian form on $T^\star_X$ induced by the curvature of the canonical bundle $\Theta_{\Lambda}(K_X)$
by contraction with the metric (we are using at this point the fact that the metric $\omega_\Lambda$ is K\"ahler). 

Given any (1,1)-form $\gamma$, we have an associate Hermitian form on $T^\star_X$ say $\Psi_\gamma$ obtained by contraction with the metric $\omega_\Lambda$. This is obtained by 
``raising the indexes", as follows. Locally near $x_0$ we write
$$\gamma= \sum_{p, q}\gamma_{p\ol q}dz_p\wedge d\ol z_q$$
and then the induced form $\Psi_\gamma$ on $T^\star_X$ is given by the expression
\begin{equation}
\Psi_\gamma= \sum_{p, q, r, s}\gamma_{p\ol q}\omega^{p\ol s}\omega^{r\ol q}
\frac{\partial}{\partial z_r}\wedge \frac{\partial}{\partial \ol z_s}.
\end{equation}

\noindent We consider the formula \ref{curb}, and we introduce the following notations. 
\medskip

\noindent $\bullet$  
$\displaystyle \Psi_{\Theta_\lambda}$ is the form induced by $\displaystyle \Theta_{\omega_\Lambda} 
(K_X)$;
\smallskip

\noindent $\bullet$ $\displaystyle \Psi_{\alpha, \varepsilon}$ is the form induced by 
$\displaystyle (1-\delta)\sum_{j=1}^N
\frac{\ep^2}{(\ep^2+ |s_j|^2)^2}\sqrt{-1}\partial s_j
\wedge \ol{\partial s_j} + \alpha  + \eta\omega_0$;
\smallskip

\noindent $\bullet$ $\displaystyle \Psi_{\Lambda}$ is the form induced by 
$$\tau_\Lambda:= \displaystyle \frac{\lambda^2}{\lambda^2+ \exp(\wh f_\eta)}\ddc \wh f_\eta
+ \sum_{j=1}^N
\Big(\delta- \frac{(1-\delta)\varepsilon^2}{\ep^2+ |s_j|^2}\Big)\theta_j+ 
\sum_{r=1}^{A_\eta} \frac{b^r_\eta\rho^2}{\rho^2+ |\sigma_{r, \eta}|^2}\beta_{r, \eta}.$$
\medskip

\noindent As we can see from the formula \ref{comp}, the curvature term in \ref{ddcm}
is expressed as follows
\begin{equation} \label{equa1}
\int_{X_0}\xi_\mu\frac{\langle\Theta_{\Lambda}(E_m)u, u\rangle}{|u|_\Lambda^2}\wedge\omega_\Lambda^{n-1}= \frac{1}{n}\int_{X_0}\xi_\mu \frac{\Psi_{\Theta_\lambda}(u, \ol u)}{|u|_\Lambda^2}\omega_\Lambda^n.
\end{equation}
where we denote 
$$\displaystyle \xi_\mu:= \xi\frac{|s|^{2m\delta}|u|_\Lambda^2}{\mu^2+ |s|^{2m\delta}|u|_\Lambda^2}\leq 1.$$

\medskip

We observe that the form $\displaystyle \Psi_{\alpha, \varepsilon}$ is positive definite (since $\alpha\geq 0$ by hypothesis), and then we have the inequalities

\begin{eqnarray}
\frac{\xi_\mu \Psi_{\Theta_\lambda}(u, \ol u)}{m|u|_\Lambda^2}\!\!\!\!& \leq & 
\!\!\!\!\frac{\xi_\mu \big(\Psi_{\Theta_\lambda}+ 
\Psi_{\alpha, \varepsilon}\big)(u, \ol u)}{m|u|_\Lambda^2}\nonumber \\
& = &\!\!\!\!\frac{\xi_\mu \big(\Psi_{\Theta_\lambda}+ \Psi_{\alpha, \varepsilon}+ \Psi_{\Lambda}\big)(u, \ol u)- \xi_\mu \Psi_{\Lambda}(u, \ol u)}{m|u|_\Lambda^2}\nonumber \\
&\leq &\!\!\!\!\tr_{\omega_\Lambda}\Big(\Theta_{\omega_\Lambda} 
(K_X)+ (1-\delta)\sum_{j=1}^N
\frac{\ep^2}{(\ep^2+ |s_j|^2)^2}\sqrt{-1}\partial s_j
\wedge \ol{\partial s_j} + \alpha  + \eta\omega_0+ \tau_\Lambda\Big)\nonumber \\
&- &\!\!\!\!\frac{\xi_\mu \Psi_{\Lambda}(u, \ol u)}{m|u|_\Lambda^2}\nonumber
\end{eqnarray}

The fact that the third inequality of the preceding relations holds true can be seen
as a consequence of the following elementary result,
combined with
the fact that the form $\displaystyle \Psi_{\Theta_\lambda}+ \Psi_{\alpha, \varepsilon}+ \Psi_{\Lambda}$ is definite positive, as the formula \ref{comp} shows it.

\begin{lem}
Let $\Theta= (\theta^p_d)$ be a positive definite Hermitian form; then we have
$$\sum_{J, r, d}u_J\ol u_{J(j_r, d)}\theta^d_{j_r}\leq m\Big(\sum_{J}|u_J|^2\Big)\big(\sum_j\theta^j_j\big)$$
for any set of complex numbers $u_J$, where $J=(j_1,\dots , j_m)\in \{1,\dots , n\}^m$ is the set of indices. 
In the relation above we denote $J(j_r, d)$ the index $(j_1,\dots j_{r-1}, d, j_{r+1},\dots , j_m)$
having the same components as $J$ except that we replace $j_r$ with $d$.
\end{lem}

\begin{proof}
For $m=1$, the inequality to be proved is
$$\sum_{j, r}u_j\ol u_r\theta^r_j\leq \big(\sum_j |u_j|^2\big)\big(\sum_j\theta^j_j\big)$$
and the easy verification will not be detailed here. For a general $m\geq 2$, we just observe that 
the corresponding sum for $r=1$ can be written as
$$\sum_{J^\prime, p, d}u_{pJ^\prime}\ol u_{dJ^\prime}\theta^d_{p} $$
where $J^\prime=(j_2,\dots , j_m)$. For each fixed index $J^\prime$, the inequality corresponding to $m=1$ 
shows that we have
$$\sum_{J^\prime, p, d}u_{pJ^\prime}\ol u_{dJ^\prime}\theta^d_{p}\leq 
\big(\sum_j |u_{jJ^\prime}|^2\big)\big(\sum_j\theta^j_j\big) $$
and summing over all the indices $J^\prime$, and then over $r=1,\dots m$ we infer the result.
\end{proof}

\noindent These considerations and (\ref{equa1}) show that we have
 
\begin{eqnarray} \label{equa2}
\int_{X_0}\xi_\mu \frac{\langle\Theta_{\Lambda}(E_m)u, u\rangle}{|u|_\Lambda^2}\wedge\omega_\Lambda^{n-1}&\leq &m\int_{X_0}
\big(\Theta_{\omega_\Lambda} 
(K_X) + \alpha \big)\wedge \omega_\Lambda^{n-1}\nonumber \\
&+ & m(1-\delta)\int_{X_0}\Big(\sum_{j=1}^N
\frac{\ep^2}{(\ep^2+ |s_j|^2)^2}\sqrt{-1}\partial s_j
\wedge \ol{\partial s_j}\Big)\wedge \omega_\Lambda^{n-1}  \\
&+& m\int_{X_0}\tau_\Lambda\wedge \omega_\Lambda^{n-1} + \eta m\int_{X_0}\omega_0^n
- \int_{X_0}\xi_\mu \Psi_{\Lambda}(u, \ol u)\omega_\Lambda^n. \nonumber
\end{eqnarray}
\smallskip

\noindent By the computations done in Step 2, we have the identity
$$(1-\delta)
\frac{\ep^2}{(\ep^2+ |s_j|^2)^2}\sqrt{-1}\partial s_j
\wedge \ol{\partial s_j}+ \delta\theta_j= (1-\delta)\ddc \log(\ep^2+ |s_j|^2)+ \theta_j- 
\frac{(1-\delta)\varepsilon^2}{\ep^2+ |s_j|^2}\theta_j$$
for each index $j$, so by integration over $X\setminus \Supp(B)$ and Stokes formula
(which indeed holds on the open manifold, since the functions/forms we are dealing with are smooth),
the inequality \ref{equa2} becomes
\begin{eqnarray} \label{equa3}
\int_{X_0}\xi_\mu\frac{\langle\Theta_{\Lambda}(E_m)u, u\rangle}{|u|_\Lambda^2}\wedge\omega_\Lambda^{n-1}&\leq &m\int_X(\Theta(K_X)+ \alpha+ \sum_{j=1}^N\theta_j)\wedge \omega_0^{n-1} \nonumber\\
&+ & m\int_{X}\frac{\lambda^2}{\lambda^2+ \exp(\wh f_\eta)}\ddc \wh f_\eta \wedge \omega_\Lambda^{n-1} \nonumber\\
&- &m\sum_{j=1}^N
\int_{X}\frac{(1-\delta)\varepsilon^2}{\ep^2+ |s_j|^2}\theta_j\wedge \omega_\Lambda^{n-1} \\
&+ & m\sum_{r=1}^{A_\eta}\int_{X}\frac{b^r_\eta\rho^2}{\rho^2+ |\sigma_{r, \eta}|^2}\beta_{r, \eta}\wedge \omega_\Lambda^{n-1}+ \eta\int_{X}\omega_0^n \nonumber\\ 
&- &\int_{X_0}\xi_\mu\Psi_{\Lambda}(u, \ol u)\omega_\Lambda^n.\nonumber
\end{eqnarray}

\noindent In the formula \ref{equa3} above, we have used the fact that the singularities of 
$\wh f_\eta$ are in codimension 2 or higher, so that we have the equality
$$\int_{X}\frac{\lambda^2}{\lambda^2+ \exp(\wh f_\eta)}\ddc \wh f_\eta \wedge \omega_\Lambda^{n-1} 
= \int_{X_0}\frac{\lambda^2}{\lambda^2+ \exp(\wh f_\eta)}\ddc \wh f_\eta \wedge \omega_\Lambda^{n-1}. $$

We derive next a upper bound for the last term in \ref{equa3}. By the expression of the 
form $\tau_\Lambda$ we see that for each $\eta> 0$ there exists a constant $C_\eta$ 
depending on $\eta$ and a constant $C$ which is uniform with respect to the set of 
parameters $\Lambda$ such that 
$$\tau_\Lambda\geq -C_\eta\Big(\frac{\lambda^2}{\lambda^2+ \exp(\wh f_\eta)}+
 \sum_r\frac{\rho^2}{\rho^2+ |\sigma_{r, \eta}|^2}\Big)\omega_0- 
 C\Big(\delta+ \sum_j\frac{\varepsilon^2}{\ep^2+ |s_j|^2}\Big)\omega_0,$$
and therefore we infer that we have
\begin{eqnarray}\label{equa4}
-\int_{X_0}\xi_\mu \Psi_{\Lambda}(u, \ol u)\omega_\Lambda^n&\leq &
C_\eta\int_{X}\Big(\delta+ \frac{\lambda^2}{\lambda^2+ \exp(\wh f_\eta)}+
\sum_r\frac{\rho^2}{\rho^2+ |\sigma_{r, \eta}|^2}\Big)
\omega_0\wedge \omega_\Lambda^{n-1}\nonumber \\
&+& C\delta+ C\sum_j\frac{\varepsilon^2}{\ep^2+ |s_j|^2}.
\end{eqnarray}

\medskip

\noindent 
\textbf{Step 4.} We summarize here the conclusion of the computations of the preceding steps.
By \ref{ddcm}, the quantity 
\begin{equation}\label{intnum}
\int_{X_0}\xi\frac{|s|^{2m\delta}|u|_\Lambda^2}{\mu^2+ |s|^{2m\delta}|u|_\Lambda^2}\Theta_{h_L}(L)\wedge\omega_\Lambda^{n-1}\end{equation}
whose limit as $\xi$ tends to the characteristic function of $X\setminus \Supp(B)$ and $\mu$ tends to zero respectively we are interested in
is smaller than
$$m\int_X(\Theta(K_X)+ \alpha+ \sum_{j=1}^N\theta_j)\wedge \omega_0^{n-1}$$
which is the bound we hope to obtain, modulo the following terms
\begin{equation} \label{equa5}
\int_{X}\frac{\varepsilon^2}{\ep^2+ |s_j|^2}\omega_0\wedge \omega_\Lambda^{n-1}, 
\end{equation}

\begin{equation} \label{equa6}
\int_{X_0}\xi \ddc\log \big(\mu^2+ |s|^{2m\delta}|u|_\Lambda^2\big)\wedge\omega_\Lambda^{n-1}
\end{equation}
as well as
\begin{equation}\label{equa7}
C_\eta\int_{X}\frac{\lambda^2}{\lambda^2+ \exp(\wh f_\eta)}\omega_0\wedge \omega_\Lambda^{n-1},
\quad C_\eta\int_{X}\frac{\rho^2}{\rho^2+ |\sigma_{r, \eta}|^2}\omega_0\wedge \omega_\Lambda^{n-1}.
\end{equation}
and
\begin{equation}\label{added}
\int_{X}\frac{\lambda^2}{\lambda^2+ \exp(\wh f_\eta)}\ddc \wh f_\eta \wedge \omega_\Lambda^{n-1}, 
\end{equation}
\smallskip

\noindent In conclusion, Theorem 1.1 will be proved if we are able to show that by some choice of the 
cut-off function $\xi$ and the parameters $\Lambda=(\ep, \lambda, \rho, \eta, \delta)$ respectively, the quantities \ref{equa5}, \ref{equa7} and \ref{added} tend to zero, and 
\ref{equa6} is negative. This will be a consequence of 
the estimates provided by Theorem \ref{MA}.

\medskip

\noindent 
\textbf{Step 5.} We first let $\varepsilon\to 0$, while the other parameters are unchanged. The effects of this first operation are evaluated in what follows. 

To start with, we recall that by \ref{estma} we have 
\begin{equation}
\frac{1}{C_{\delta, \eta}}\omega_{\Lambda}\leq \sum_{j=1}^p\frac{\sqrt -1dz^j\wedge dz^{\ol j}}{(\varepsilon^2+ |z^j|^2)^{1-\delta}}+ \sum_{j=p+1}^n\sqrt -1dz^j\wedge dz^{\ol j}\nonumber
\end{equation}
locally at each point of $X$, where $(z_j)$ are coordinates adapted to the pair $(X, B)$.
As a consequence we infer that
\begin{equation}\label{equa8}
\lim_{\ep \to 0}\int_{X}\frac{\varepsilon^2}{\ep^2+ |s_j|^2}\omega_0\wedge \omega_\Lambda^{n-1}
= 0 
\end{equation}
by a quick computation which will not be detailed here. 

We will analyze next the quantity \ref{equa6} as $\ep\to 0$. Let $\Lambda_0:= (0, \lambda, \rho, \eta, \delta)$ be the new set of parameters. We know that 
$$\omega_\Lambda\to \omega_{\Lambda_0}$$
uniformly on compact sets of $X_0$. Therefore we have
\begin{equation} \label{equa9}
\lim_{\ep \to 0}\int_{X_0}\xi \ddc\log \big(\mu^2+ |s|^{2m\delta}|u|_\Lambda^2\big)\wedge\omega_\Lambda^{n-1}=
\int_{X_0}\xi \ddc\log \big(\mu^2+ |s|^{2m\delta}|u|_{\Lambda_0}^2\big)\wedge\omega_{\Lambda_0}^{n-1}
\end{equation}
The important difference between $|u|_\Lambda$ and $\displaystyle |u|_{\Lambda_0}$
is that the pole order of the latter quantity along the components of $B$ is smaller than $m\delta$. 
Indeed we have
\begin{equation}\label{equa10}
\sum_{j=1}^p
\frac{\sqrt -1dz^j\wedge dz^{\ol j}}{|z^j|^{2(1-\delta)}}+ \sum_{j=p+1}^n\sqrt -1dz^j\wedge dz^{\ol j} \leq 
\frac{1}{C_{\lambda, \rho, \delta, \eta}}\omega_{\Lambda_0} 
\end{equation}
as well as 
\begin{equation}\label{equa11}
\frac{1}{C_{\delta, \eta}}\omega_{\Lambda_0}\leq 
\sum_{j=1}^p
\frac{\sqrt -1dz^j\wedge dz^{\ol j}}{|z^j|^{2(1-\delta)}}+ \sum_{j=p+1}^n\sqrt -1dz^j\wedge dz^{\ol j}
\end{equation}
by combining \ref{estma} with the Monge-Amp\`ere equation verified by 
$\omega_\Lambda$. The upshot is that the function
$$\log \big(\mu^2+ |s|^{2m\delta}|u|_{\Lambda_0}^2\big)$$
is bounded by a constant depending on $\mu, \Lambda_0$, but completely independent with respect to the size of the support of $\xi$.
\medskip

\noindent At this stage or our proof, we choose a sequence of cut-off functions $$\xi:= \chi_\tau$$
as in \cite{CGP} converging to the characteristic function of the set 
$X\setminus \Supp(B)$, so that the quantity \ref{equa6} becomes
\begin{equation} \label{equa12}
\int_{X_0}\log \big(\mu^2+ |s|^{2m\delta}|u|_{\Lambda_0}^2\big)\ddc \chi_\tau\wedge\omega_{\Lambda_0}^{n-1}.
\end{equation}
Indeed, the integration by parts is legitimate, since for every positive $\tau$, the relevant quantities are non-singular on the support of $\chi_\tau$.

If we let $\tau\to 0$, the quantity \ref{equa12} tends to zero, thanks to the computations e.g. in \cite{CGP}, together with the fact that the bound of the absolute value of the function
$\log \big(\mu^2+ |s|^{2m\delta}|u|_{\Lambda_0}^2\big)$ is independent of $\tau$.

Also, we remark that the term \ref{intnum} becomes simply
\begin{equation}
\label{inter}
\int_{X}\frac{|s|^{2m\delta}|u|_{\Lambda_0}^2}{\mu^2+ |s|^{2m\delta}|u|_{\Lambda_0}^2}\Theta_{h_L}(L)\wedge\omega_{\Lambda_0}^{n-1}.
\end{equation}

\noindent As a conclusion of this step, by first letting $\ep\to 0$ and then $\xi \to \chi_{X_0}$ as indicated above, we infer that we have
\begin{eqnarray}\label{step5}
\int_{X}\frac{|s|^{2m\delta}|u|_\Lambda^2}{\mu^2+ |s|^{2m\delta}|u|_{\Lambda_0}^2}\Theta_{h_L}(L)\wedge\omega_{\Lambda_0}^{n-1}&\leq & m\int_X(\Theta(K_X)+ \alpha+ \sum_{j=1}^N\theta_j)\wedge \omega_0^{n-1}\nonumber \\
&+ & C_\eta \Big(\int_{X}\frac{\lambda^2}{\lambda^2+ \exp(\wh f_\eta)}\omega_0\wedge \omega_{\Lambda_0}^{n-1}\Big)\nonumber \\ 
&+ & C_\eta \Big(
\sum_r\int_{X}\frac{\rho^2}{\rho^2+ |\sigma_{r, \eta}|^2}\omega_0\wedge \omega_{\Lambda_0}^{n-1}
\Big)\\
&+ & C\int_{X}\frac{\lambda^2}{\lambda^2+ \exp(\wh f_\eta)}\ddc \wh f_\eta \wedge \omega_{\Lambda_0}^{n-1}+ (\delta+ \eta)C.\nonumber 
\end{eqnarray}
We let $\mu\to 0$; the left-hand side term in \ref{step5} becomes
$$\int_{X}\Theta_{h_L}(L)\wedge\omega_0^{n-1}$$
as we see by dominated convergence theorem.
\medskip

\noindent 
\textbf{Step 6.} It is the last step of our proof; we first establish the equalities
\begin{equation}\label{equa17}
\lim_{\lambda\to 0}\int_{X}\frac{\lambda^2}{\lambda^2+ \exp(\wh f_\eta)}\omega_0\wedge \omega_{\Lambda_0}^{n-1}= 0
\end{equation}
and
\begin{equation}\label{equa18}
\lim_{\rho\to 0}\int_{X}\frac{\rho^2}{\rho^2+ |\sigma_{r, \eta}|^2}\omega_0\wedge \omega_{\Lambda_0}^{n-1}= 0
\end{equation}
for any set of parameters $\delta$ and $\eta$. This is quite easy: by the relation 
\ref{equa11} we have 
$$\int_{X}\frac{\lambda^2}{\lambda^2+ \exp(\wh f_\eta)}\omega_0\wedge \omega_{\Lambda_0}^{n-1}
\leq C_{\delta, \eta}\int_{X}\frac{\lambda^2}{\lambda^2+ \exp(\wh f_\eta)}
\frac{dV}{\prod_j|s_j|^{2-2\delta}}$$
and \ref{equa17}, \ref{equa18} follow by dominated convergence theorem.
\medskip

%%%%%%%%%%%%%%

We treat next the remaining term; we claim that the next inequality holds
\begin{equation}\label{equa19}
\lim_{\lambda\to 0}\int_{X}\frac{\lambda^2}{\lambda^2+ \exp(\wh f_\eta)}\ddc \wh f_\eta \wedge \omega_{\Lambda_0}^{n-1}= 0
\end{equation}
for any set of parameters $(\eta, \delta)$ within the range we have fixed at the beginning of the proof. 

The relation \ref{equa19} is established as follows:
\begin{equation}\label{equa20}
\int_{X}\frac{\lambda^2}{\lambda^2+ \exp(\wh f_\eta)}\ddc \wh f_\eta \wedge \omega_{\Lambda_0}^{n-1}
\leq 
C_{\eta, \delta}\int_{X}\frac{\lambda^2}{\lambda^2+ \exp(\wh f_\eta)}
(\ddc \wh f_\eta + C_\eta\omega_{\delta, \rm st})\wedge \omega_{\delta, \rm st}^{n-1}
\end{equation}
where $C_\eta$ is a constant such that the (1,1) form 
$\ddc \wh f_\eta + C_\eta\omega_{\delta, \rm st}$ is definite positive. We use the same notations as before, namely $\omega_{\delta, \rm st}$ is the standard metric with conic singularities 
along $B$, with all cone angles equal to $2\pi(1-\delta)$.

We define 
$$d\mu_{\eta, \delta}:= (\ddc \wh f_\eta + C_\eta\omega_{\delta, \rm st})\wedge \omega_{\delta, \rm st}^{n-1};$$
it is a positive measure of finite mass. Let $p:X_\eta\to X$ be a birational map such that 
$$p^\star\big(\ddc \wh f_\eta + C_\eta\omega_{\delta, \rm st}\big)= [E]+ \gamma_\eta$$
where $E$ is a divisor and $\gamma_\eta$ is smooth. The important information about the 
singularities of $\wh f_\eta$ is the $E$ is $p$-contractible. Thus we have
$$p^\star d\mu_{\eta, \delta}= \gamma_\eta\wedge p^\star \omega_{\delta, \rm st}^{n-1}:= 
d\wh \mu_{\eta, \delta}$$
that is to say, a measure with mild singularities on $X_\eta$. In fact, an immediate calculation shows
that the measure $\displaystyle d\wh \mu_{\eta, \delta}$ is smaller than the determinant of a standard 
metric with conic singularities of angles $2\pi(1-\delta)$ along the support of the divisor $p^\star(B)$,  
modulo
some constant \emph{independent of $\lambda$}.
The quantity we have to analyze becomes
$$\int_{X_\eta}\frac{\lambda^2}{\lambda^2+ \exp(\wh f_\eta\circ p)}d\wh \mu_{\eta, \delta}$$
and indeed it converges to zero if $\lambda\to 0$ by the arguments already invoked before
(the dominate convergence theorem). The assertion is therefore established. 

\begin{rema} We remark 
that the hypothesis concerning the codimension of the singularities of $\wh f_\eta$ is essential. 
Indeed, if $v$ is a holomorphic section of an ample line bundle and if 
we take e.g. $\log |v|^2$ instead of $\wh f_\eta$, we see that \ref{equa19} is not verified!
\end{rema}

\medskip

\noindent After this last operation, the inequality \ref{step5} becomes
\begin{equation}\label{equa21}
\int_{X}\Theta_{h_L}(L)\wedge\omega_0^{n-1}\leq (\delta+ \eta)C+ m\int_X(\Theta(K_X)+ \alpha+ \sum_{j=1}^N\theta_j)\wedge \omega_0^{n-1}
\end{equation}
and Theorem 1.1 is proved, by taking $\delta, \eta\to 0$.

\begin{rema} It is possible to adapt the previous argument to the case where 
$\alpha$ is possibly singular, i.e. the class $\{\alpha\}$ is psef rather than Hermitian semi-positive. However, we have to impose the condition that the 
generic Lelong number of $\alpha$ along each component of $B$ is equal to zero. We leave the
details to the interested reader.
\end{rema}

\begin{rema} 
An important application of our arguments was obtained in \cite{G2}, where 
the stability of log-tangent bundle with respect to $K_X+ B$ is proved, under the hypothesis that this latter bundle is nef and big. 
\end{rema}

\section{Proof of Theorem 1.2}
\medskip

In this part of our paper we will prove Theorem 1.2. As we have already mentioned, the result itself is already known: its proof in \cite{CP} is based on a the generic semi-positivity result for the cotangent of the orbifold pairs whose canonical bundle is pseudo-effective. The proof we present here
follows essentially the same ideas, modulo the fact that we are using Theorem 1.1 instead of the aforementioned 
result concerning the orbifold cotangent bundle.

\subsection{Continuity method} 

For the rest of this paper the manifold $X$ is assumed to be projective. Let $A$ be an ample
on $X$ such that $K_X+ B+ A$ is ample.
We consider the interval 
\begin{equation}
J:= \{t\in [0, 1] : K_X+ B+ tA \hbox{ is big }\}
\end{equation}
It is clear that $J$ is non-empty and open. We show next that $J$ is equally closed, provided that 
there exists an injection
$$L\to \TB$$
for some $m> 0$, where $L$ is a big line bundle on $X$.

Let $(t_k)\subset J$, converging to a real number $t_\infty$. We have to show that the limit  $\R$-bundle 
$K_X+ B+ t_\infty A$ is big. The first observation is that this bundle is at least pseudo-effective, given that it is a limit of big bundles. 
\medskip

As a warm-up, we first discuss the a very particular case, namely \emph{we assume that the limit 
$K_X+ B+ t_\infty A$ is nef}. Then we argue as follows: for each $k$ the bundle $K_X+ B+ t_k A$ 
is ample, and by Theorem 1.1 we have the numerical inequality
\begin{equation}\label{equa41}
\int_Xc_1(L)\wedge c_1(K_X+ B+ t_k A)^{n-1}\leq m\int_Xc_1(K_X+ B+ t_k A)^{n},
\end{equation}
where we stress on the fact that $m$ is a purely numerical constant, 
in particular it is independent of $k$.

Since $L$ is big, we certainly have
$$\int_Xc_1(L)\wedge c_1(K_X+ B+ t_k A)^{n-1}\geq C(L)
\int_Xc_1(A)\wedge c_1(K_X+ B+ t_k A)^{n-1}$$
for some constant $C(L)$ depending exclusively on $L$. 

On the other hand, by the Hovanski-Teissier concavity inequality, we have
\begin{equation}\label{equa42}
\int_Xc_1(A)\wedge c_1(K_X+ B+ t_k A)^{n-1}\geq \Big(\int_Xc_1(A)^n\Big)^{\frac{1}{n}}
\Big(\int_Xc_1(K_X+ B+ t_k A)^{n}\Big)^{\frac{n-1}{n}}
\end{equation}
By \ref{equa41} and \ref{equa42} we infer that
$$\int_Xc_1(K_X+ B+ t_k A)^{n}\geq C_0> 0$$
for some constant $C_0$ which is independent of $k$; hence, the same will be true for the limit,
so the $\R$ bundle $K_X+ B+ t_\infty A$ is big, and the proof of the particular case is finished.
\medskip

We will prove the general case of Theorem 1.2 along the same line of arguments. 
The bundle $K_X+ B+ t_\infty A$ is no longer assumed to be nef, but nevertheless, it is at least 
pseudo-effective, and the idea is to decompose it into two orthogonal pieces, a nef part and an effective part. For an arbitrary big line bundle, this cannot be done in a sufficiently accurate way
so as to be relevant for us (the approximate Zariski decomposition is not useful here,
as we will comment at the end of this paper). However, the bundle we are dealing with is 
an \emph{adjoint bundle}
and the additional techniques required to treat this case are explained in the following paragraphs.

\subsection{Desingularisation and Zariski decomposition} In this section we collect a few important facts 
concerning the principalization of the ideals of holomorphic functions. They will be used in conjunction with
the finite generation result in \cite{BCHM}.

Let $\cI\subset \cO_X$ be an ideal sheaf, and let $x\in X$ be a point. The \emph{vanishing order} of $\cI$ at $x$ is defined as follows
$$\ord_x(\cI):= \max \big\{r : \cI\subset m_x^r\big\}$$
where $m_x$ is the ideal sheaf of $x$. Given $Z\subset X$ a submanifold, \emph{the order of $\cI$ along $Z$} is defined as
the vanishing order of $\cI$ at the generic point of $Z$. Finally, \emph{the maximal order of $\cI$ along $Z$} is equal to
the maximum of the numbers $\ord_x(\cI)$ for all $x\in Z$.
\smallskip

A \emph{marked ideal} $(\cI, m)$ is a couple consisting of an ideal sheaf $\cI$
together with a positive integer $m$. Let $Z$ be a smooth subvariety of $X$ such that the order of $\cI$ along $Z$
is at least $m$. The inverse image of the ideal $\cI$ with respect to  
the blow-up $\pi: \widehat X\to X$ of $X$ along $Z$ can we written as the product of the principal ideal corresponding to the 
exceptional divisor to the power $m$, multiplied with an ideal of holomorphic functions on $\widehat X$. The latter is called 
\emph{proper transform of the marked ideal $(\cI, m)$}.

Let $E:=(E_1,\dots, E_s)$ be a set of non-singular hypersurfaces of $X$, such that 
$\sum_jE_j$ is a simple normal crossing divisor. Following \cite{JK1}, we call $(X, \cI, m, E)$ a 
\emph{triple}.

\smallskip

 A \emph{smooth blow-up sequence of order $\geq m$} of the triple $(X, \cI, m, E)$ is a sequence 
$(X_j, \cI_j, m, E_{(j)})_{j=0,..., r}$ verifying the following requirements.

\begin{enumerate}

\item[{\rm (1)}] For each $j=0,..., r$ we have $\displaystyle \cI_j\subset \cO_{X_j}$.
\smallskip

\item[{\rm (2)}] We have $(X_0, \cI_0, m, E_{(0)}):= (X, \cI, m, E)$ and for each $j=1,..., r$ the map $X_j\to X_{j-1}$ is the blow-up 
of a smooth center $Z_{j-1}\subset X_{j-1}$ such that $\displaystyle \ord_{Z_{j-1}}(\cI_{j-1})\geq m$, and such that $Z$ has only simple normal crossings with the components of $E_{(j)}$.
\smallskip

\item[{\rm (3)}] For each $j=1,..., r$ the ideal $\cI_j$ is the proper transform of the marked ideal $(\cI_{j-1}, m)$
with respect to the map $X_{j}\to X_{j-1}$.
\smallskip

\item[{\rm (4)}] For each $j=1,..., r$, the set of hypersurfaces $E_{(j)}$ is the birational transform of $E$,
together with the exceptional divisor of the map $X_j\to X_{j-1}$.

\end{enumerate}

\medskip

\noindent
In this context, we quote a result from \cite{JK1}, page 41. 

\begin{theo} [\cite{JK1}] Let $X$ be a smooth projective manifold, and let $0\neq \cI\subset \cO_X$ be an ideal sheaf.
Let $m$ be a positive integer assumed to be smaller that the maximal order of $\cI$ along $X$, and let $E$ be a simple normal crossing divisur. Then there exists a 
smooth blow-up sequence of order $\geq m$ of the triple $(X, \cI, m, E)$ such that the proper transform of $\cI$ by the resulting map $\widehat X\to X$
is an ideal whose maximal order along $\widehat X$ is strictly smaller than $m$.
\end{theo}

\bigskip

\noindent As a consequence of the algorithm described in the previous theorem (with $m=1$) we have the following result.

\begin{theo} \label{prince} Let $(X, B)$ be a smooth log-canonical pair,
and let $0\neq \cI\subset \cO_X$ be an ideal sheaf.
There exists a birational map $p: X_1\to X$ such that the inverse image of $\cI$ with respect to $p$ is a principal ideal, 
whose zero loci plus the inverse image of the boundary divisor $B$ is a normal crossing divisor. 
Moreover, $p$ chosen so that the next property is satisfied.

\noindent {\rm \bf (Supp)} Let $x\in X$ be an arbitrary point, and let $(f_1,..., f_g)$ be the local generators of 
the ideal $\cI$ on an open set $U$ centered at $x$. Then the support of the relative canonical bundle
$K_{X_1/X}$ intersected with the inverse image of $U$ 
is contained in the set $\displaystyle \bigcap_{j=1}^g \big(f_j\circ p = 0\big).$

\end{theo}

\bigskip

%%%%%%%%%%%%%%%%%%%%%%%%%%%%%%%%%%%%%%%%%%%%%%%%%%%%%%%%%%%%%%%%%%%%%%%%%%

\noindent We are considering next the context of the adjoint bundles, in which 
we have the following important result. This represents the second main technical point in the proof of 
Theorem 1.2.

\begin{theo}[\cite{BCHM}] For each $k\geq 1$, 
the algebra ${\mathcal R}_k$ of pluricanonical sections corresponding to $K_X+ B+ t_kA$ is finitely generated. In particular, the 
$\Q$-bundle $K_X+ B+ t_kA$ admits a Zariski decomposition--obtained by considering a desingularisation of the
ideal defined by the generators of ${\mathcal R}_k$.
\end{theo}

\smallskip

\noindent We consider the generators $\displaystyle \big(u^{(k)}_j\big)_{j= 1,... g_k}$ of the algebra ${\mathcal R}_k$; we assume that they are sections of the bundle $m_k\big(K_X+ B+ t_kA\big)$. Let 
$T_k$ be the curvature current of the metric on $K_X+ B+ t_kA$ induced by the generators above. There exists a modification of
$X$, say 
$p_k: X_k\to X$ such that 
$$p_k^\star(T_k)= \omega_k+ [N_k]$$
where $\omega_k$ is a semi-positive (1,1)-form on $X_k$ corresponding to the $\Q$-bundle $P_k$ on $X_k$, and $N_k$ is an effective divisor, such that 
$$\Vol(K_X+ B+ t_kA)= \int_{X_k}\omega_k^n.$$
The finiteness of the algebra ${\mathcal R}_k$ reflects into the \emph{equality} above--in general, we only have an approximation of the volume by the top power of $\omega_k$. 
\medskip

\noindent Combined with Theorem \ref{prince} above, we obtain the next statement.

\begin{cor} For each $k\geq 1$ there exists a non-singular manifold $X_k$ and a birational map $p_k: X_k\to X$ such that the following holds true. 

\begin{enumerate}

\item[{\rm (a)}] We have 
$p_k^\star(K_X+ B+ t_kA)= P_k+ N_k,$ where $N_k$ is effective, 
$P_k$ is big, without base points, and such that
$$P_k^n= \Vol(K_X+ B+ t_kA).$$
\smallskip

\item[{\rm (b)}] The orthogonality relation holds true $P_k^{n-1}\cdot N_k= 0.$
\smallskip

\item[{\rm (c)}] We have $\displaystyle \Supp\big(K_{X_k/X}\big)\subset \Supp (N_k)$.

\end{enumerate}

\end{cor}

\medskip

\noindent For the crucial information (b) we refer to \cite{BDPP}, orthogonality lemma. The point (c) is a direct consequence of 
Theorem \ref{prince} discussed above.
\smallskip

\subsection{End of the proof}

We denote by $\displaystyle \big(\Xi_k^{(j)}\big)_{j=1,..., i_k}$ the support of the exceptional divisor of the map
$p_k$. By the inclusion (c) above together with the orthogonality relation (b) of the Corollary 1.2 we obtain
\begin{equation}\label{equa44}
P_k^{n-1}\cdot \Xi_k^{(j)}= 0\end{equation}
for each $j= 1,..., i_k$. This is an important fact, since it will allow us to ``ignore" the presence of the 
exceptional divisors, and argue in what follows basically as in the case nef explained at the beginning of our proof.
\smallskip

Let $B= Y_1+...+ Y_N$; we have the change of variables formula 
\begin{equation}\label{equa45}
p_k^\star(K_X+ B)= K_{X_k}+ B_k- E_k^2\end{equation}
where $B_k:= \sum_{j=1}^N\ol Y_j+ E_k^1$; here we denote by $\ol Y_j$ the proper transform of $Y_j$, and $E_k^1, E_k^2$ are effective divisors, such that their support 
is contained in the exceptional divisor of $p_k$. The divisor $E_k^1$ is reduced, thanks to the fact that $(X, B)$ is an lc pair
and moreover we have 
\begin{equation}\label{equa46}
\Supp(B_k)\subset \Supp\big(p_k^\star(B)\big).\end{equation}
\smallskip

By hypothesis, we have an injection
$$L\to \TB$$
where $L$ is a big line bundle. We consider the $p_k$-inverse image of this injection, and we get
\begin{equation}\label{equa47}
p_k^\star(L)\to \otimes^mT^\star_{X_k}\langle \ol B_k\rangle\end{equation}
where $\ol B_k$ is the reduced part of the inverse image of the divisor $p_k^\star(B)$.

\smallskip

The bundle 
$K_{X_k}+ \ol B_k+ t_\infty p_k^*A$ is pseudo-effective, as it follows from
\ref{equa45} combined with \ref{equa46}. Indeed, as a consequence of these relations we infer that 
the difference $\ol B_k- B_k$ is effective and $p_k$-exceptional.

Moreover, the bundle $P_k$ is nef so by Theorem 1.1
we obtain the inequality
\begin{equation}\label{equa48}
\int_{X_k}p_k^\star(L)\cdot P_k^{n-1}\leq m\int_{X_k}(K_{X_k}+ \ol B_k+ t_0p_k^*A)\cdot P_k^{n-1}.\end{equation}
 
\smallskip

By the relations \ref{equa46} and \ref{equa45}, together with the definition of $B_k$ we infer that
\begin{equation}\label{equa49}(K_{X_k}+ \ol B_k+ t_\infty p_k^*A)\cdot P_k^{n-1}= p_k^\star(K_X+ B+ t_\infty A)\cdot P_k^{n-1}.\end{equation}
The equalities (a) and (b) in Corollary 1 show that we have
\begin{equation}\label{equa50}p_k^\star(K_X+ B+ t_\infty A)\cdot P_k^{n-1}= P_k^{n}- 
(t_k-t_\infty)p_k^\star(A) \cdot P_k^{n-1}\end{equation}
Since the bundle $L$ is big, we certainly have $\displaystyle p_k^\star(L)\cdot P_k^{n-1}\geq \varepsilon_0 p_k^\star(A) \cdot P_k^{n-1}$ for some positive $\varepsilon_0$, so all in all we obtain
\begin{equation}\label{equa51}P_k^{n}\geq mp_k^\star(A) \cdot P_k^{n-1}\end{equation}
We use the Khovanskii-Teissier inequality as in the ``nef" case discussed previously, and as a consequence we get
\begin{equation}\label{equa52}p_k^\star(A) \cdot P_k^{n-1}\geq \big(A^n\big)^{1\over n}\cdot \big(P_k^{n}\big)^{n-1\over n}.\end{equation}
The conjunction of \ref{equa51} and \ref{equa52} gives
\begin{equation}\label{equa53}P_k^n\geq C(m)>0\end{equation}
uniformly with respect to $k$, so Theorem 1.2 is proved.\qed

\section{Further remarks}

We consider the following exact sequence 
\begin{equation}\label{equa61}
0\to L\to \otimes ^mT^\star_X\la B\ra \to Q\to 0 
\end{equation}
where $(X, B)$ is a pair with the properties stated in Theorem 1.1, and $L$ is a line bundle. Let $\{\alpha\}$
be a semi-positive class of (1,1)-type, such that $c_1(K_X+ B)+ \{\alpha\}$ is psef. 

\noindent As a consequence of Theorem 1.1, we infer that we have the inequality
\begin{equation}\label{equa62}
\int_Xc_1(Q)\wedge \omega^{n-1}\geq \Big(\frac{n^m-1}{n-1}- 
m\Big)\int_X(c_1(K_X+ B)+ \{\alpha\})\wedge \omega^{n-1}-\frac{n^m-1}{n-1}\int_X\alpha\wedge \omega^{n-1}
\end{equation}
for any class $\omega$ belonging to the closure of the K\"ahler cone of $X$.

We remark that even if $\alpha=0$, so that $K_X+B$ is psef, the inequality \ref{equa62} is more precise than 
the one derived from Y. Miyaoka's result, cf. \cite{Miyaoka} as soon as $m\geq 2$. 

\noindent The next result is an easy consequence of (\ref{equa62}).

\begin{theo}
Under the hypothesis above, we assume that we have 
$$\displaystyle \int_Xc_1(Q)\wedge \omega^{n-1}=0$$
for some K\"ahler metric $\omega$. Then $K_X+B$ is numerically trivial.
\end{theo}

\medskip

If $L$ in (\ref{equa61}) is a sheaf of arbitrary rank, say $r$, then we obtain a weaker inequality
\begin{equation}\label{equa63}
\int_Xc_1(Q)\wedge \omega^{n-1}\geq -\frac{n^m-1}{n-1}\int_X\alpha\wedge \omega^{n-1}
\end{equation}
Other applications will be discussed in a forthcoming publication.

\mainmatter

\backmatter

\bibliographystyle{smfalpha}
\bibliography{biblio}

\providecommand{\bysame}{\leavevmode ---\ }
\providecommand{\og}{``}
\providecommand{\fg}{''}
\providecommand{\smfandname}{\&}
\providecommand{\smfedsname}{\'eds.}
\providecommand{\smfedname}{\'ed.}
\providecommand{\smfmastersthesisname}{M\'emoire}
\providecommand{\smfphdthesisname}{Th\`ese}
\begin{thebibliography}{BCHM10}

\bibitem[BCHM10]{BCHM}
{\scshape C.~Birkar, P.~Cascini, C.~Hacon {\normalfont \smfandname}
  J.~McKernan} -- {\og {Existence of minimal models for varieties of log
  general type}\fg}, \emph{J. Amer. Math. Soc.} \textbf{23} (2010),
  p.~405--468.

\bibitem[BDPP13]{BDPP}
{\scshape S.~Boucksom, J.-P. Demailly, M.~P{\u a}aun {\normalfont \smfandname}
  T.~Peternell} -- {\og {The pseudo-effective cone of a compact K{\"a}hler
  manifold and varieties of negative Kodaira dimension}\fg}, \emph{J. Alg.
  Geometry} (2013).

\bibitem[CGP13]{CGP}
{\scshape F.~Campana, H.~Guenancia {\normalfont \smfandname} M.~P\u{a}un} --
  {\og {Metrics with cone singularities along normal crossing divisors and
  holomorphic tensor fields}\fg}, \emph{Ann. Scient. Éc. Norm. Sup.}
  \textbf{46} (2013), p.~879--916.

\bibitem[CP13]{CP}
{\scshape F.~Campana {\normalfont \smfandname} M.~P{\u{a}}un} -- {\og Orbifold
  generic semi-positivity: an application to families of canonically polarized
  manifolds\fg}, \emph{arXiv:1303.3169} (2013).

\bibitem[Dem82]{D1}
{\scshape J.-P. Demailly} -- {\og Estimations {$L^{2}$} pour l'op{\'e}rateur
  {$\bar \partial $} d'un fibr{\'e} vectoriel holomorphe semi-positif au-dessus
  d'une vari{\'e}t{\'e} k{\"a}hl{\'e}rienne compl{\`e}te\fg}, \emph{Ann. Sci.
  {\'E}cole Norm. Sup. (4)} \textbf{15} (1982), no.~3, p.~457--511.

\bibitem[Dem92]{D2}
\bysame , {\og Regularization of closed positive currents and intersection
  theory\fg}, \emph{J. Algebraic Geom.} \textbf{1} (1992), no.~3, p.~361--409.

\bibitem[Don85]{Don85}
{\scshape S.~Donaldson} -- {\og {Anti self-dual Yang Mills connections over
  complex algebraic surfaces and stable vector bundles}\fg}, \emph{Proc. Lond.
  Math. Soc., III. Ser.} \textbf{50} (1985), p.~1--26.

\bibitem[Eno88]{Enoki}
{\scshape I.~Enoki} -- {\og Stability and negativity for tangent sheaves of
  minimal {K}{\"a}hler spaces\fg}, in \emph{Geometry and analysis on manifolds
  ({K}atata/{K}yoto, 1987)}, Lecture Notes in Math., vol. 1339, Springer,
  Berlin, 1988, p.~118--126.

\bibitem[Gue13]{G2}
{\scshape H.~Guenancia} -- {\og {Semi-stability of the tangent sheaf of
  singular varieties}\fg}, \emph{available on the author's web page} (2013).

\bibitem[Kob87]{Koba}
{\scshape S.~Kobayashi} -- \emph{{Differential geometry of complex vector
  bundles.}}, {Princeton, NJ: Princeton University Press; Tokyo: Iwanami Shoten
  Publishers}, 1987 (English).

\bibitem[Kol07]{JK1}
{\scshape J.~Koll\'ar} -- \emph{{Lectures on Resolution of Singularities.}},
  {Annals of Mathematical Studies}, 2007 (English).

\bibitem[Miy87]{Miyaoka}
{\scshape Y.~Miyaoka} -- {\og {The Chern classes and Kodaira dimension of a
  minimal variety.}\fg}, {}, 1987 (English).

\bibitem[Siu87]{Siu}
{\scshape Y.-T. Siu} -- \emph{{Lectures on Hermitian-Einstein Metrics for
  Stable Bundles and K{\"a}hler-Einstein Metrics}}, Birkh{\"a}user, 1987.

\bibitem[Yau78]{Yau78}
{\scshape S.-T. Yau} -- {\og {On the Ricci curvature of a compact K{\"a}hler
  manifold and the complex Monge-Amp{\`e}re equation. I.}\fg}, \emph{Commun.
  Pure Appl. Math.} \textbf{31} (1978), p.~339--411.

\end{thebibliography}


@Article{ BG,
	hyphenation = "american",
	author = "Robert J. Berman and Henri Guenancia",
	title = {{K{\"a}hler-Einstein metrics on stable varieties and log canonical pairs}},
	date = "2013-04-07",
	year = "2013",
	eprinttype = "arxiv",
	journal = "arXiv:1304.2087",
	eprint = "1304.2087"
}

@Article{ DZD,
	AUTHOR = "S{\'e}bastien Boucksom",
	TITLE = "Divisorial {Z}ariski decompositions on compact complex manifolds",
	JOURNAL = "Ann. Sci. {\'E}cole Norm. Sup. (4)",
	FJOURNAL = "Annales Scientifiques de l'{\'E}cole Normale Sup{\'e}rieure. Quatri{\`e}me S{\'e}rie",
	VOLUME = "37",
	YEAR = "2004",
	NUMBER = "1",
	PAGES = "45--76"
}

@InCollection{ Miyaoka,
	author = "Yoichi Miyaoka",
	title = "{The Chern classes and Kodaira dimension of a minimal variety.}",
	language = "English",
	publisher = "{}",
	year = "1987"
}

@article{CP,
hyphenation = {american},
author = {Fr{\'e}d{\'e}ric Campana and Mihai P{\u{a}}un},
title = {Orbifold generic semi-positivity: an application to families of canonically polarized manifolds},
date = {2013-04-28},
year = {2013},
journal = {arXiv:1303.3169},
} 

@Article{ Cao,
	hyphenation = "american",
	author = "Junyan Cao",
	title = {A remark on compact K{\"a}hler manifolds with nef anticanonical bundles and its applications},
	date = "2013-05-19",
	year = "2013",
	eprinttype = "arxiv",
	journal = "arXiv:1305.4397"
}

@Article{ Pog,
	author = "A.V. Pogorelov",
	title = "{The Dirichlet problem for the n-dimensional analogue of the Monge-Amp{\`e}re equation.}",
	language = "English. Russian original",
	journal = "Sov. Math., Dokl. ",
	volume = "12",
	pages = "1727--1731",
	year = "1971",
	classmath = "{*35Q99 (PDE of mathematical physics and other areas) 35D10 (Regularity of generalized solutions of PDE) }"
}

@book { SC,
    AUTHOR = {Saloff-Coste, Laurent},
     TITLE = {Aspects of {S}obolev-type inequalities},
    SERIES = {London Mathematical Society Lecture Note Series},
    VOLUME = {289},
 PUBLISHER = {Cambridge University Press},
   ADDRESS = {Cambridge},
      YEAR = {2002},
     PAGES = {x+190},
      ISBN = {0-521-00607-4},
   MRCLASS = {46E35 (35B65 35J15 35K10 46N20 58J05)},
  MRNUMBER = {1872526 (2003c:46048)},
MRREVIEWER = {Ji{\v{r}}{\'{\i}} R{\'a}kosn{\'{\i}}k},
}


@misc{ Krummel,
  Author = {Brian Krummel},
  Institution = {Cambridge University},
  Howpublished = {Cambridge University Lecture, available at the author's webpage},
  Year = {2012},
  Title = {{DeGiorgi-Nash lecture notes}}
}


@Article{ BFJ2,
	hyphenation = "american",
	author = "S{\'e}bastien Boucksom and Charles Favre and Mattias Jonsson",
	title = "A refinement of Izumi's Theorem",
	date = "2012-09-18",
	year = "2012",
	eprinttype = "arxiv",
	journal = "arXiv:1209.4104"
}

@Misc{ Blo,
	author = "Zbigniew B{\l}ocki",
	title = "{On the regularity of the complex Monge-Amp{\`e}re operator.}",
	language = "English",
	howpublished = "{Kim, Kang-Tae (ed.) et al., Complex geometric analysis in Pohang. POSTECH-BSRI SNU-GARC international conference on several complex variables, Pohang, Korea, June 23-27, 1997. Providence, RI: American Mathematical Society. Contemp. Math. 222, 181-189 (1999).}",
	year = "1999",
	abstract = "{The author applies some results on fully nonlinear elliptic operators to the theory of the complex Monge-Amp{\`e}re operator.}",
	reviewer = "{Viorel V{\^a}j{\^a}itu (Bucure\c{s}ti)}",
	keywords = "{plurisubharmonic function, nonlinear elliptic operator, complex Monge-Amp{\`e}re operator}",
	classmath = "{*32W20 (Complex Monge-Amp{\`e}re operators) 32V05 (CR structures etc.) 35J60 (Nonlinear elliptic equations) }"
}

@Article{ El1,
	AUTHOR = "R. Elkik",
	TITLE = "Fibr{\'e}s d'intersections et int{\'e}grales de classes de {C}hern",
	JOURNAL = "Ann. Sci. {\'E}cole Norm. Sup. (4)",
	FJOURNAL = "Annales Scientifiques de l'{\'E}cole Normale Sup{\'e}rieure. Quatri{\`e}me S{\'e}rie",
	VOLUME = "22",
	YEAR = "1989",
	NUMBER = "2",
	PAGES = "195--226",
	ISSN = "0012-9593",
	CODEN = "ASENAH",
	MRCLASS = "14C17 (14F05)",
	MRNUMBER = "1005159 (90j:14010)",
	MRREVIEWER = "P. E. Newstead",
	URL = "http://www.numdam.org/item?id=ASENS_1989_4_22_2_195_0"
}

@Article{ R-Z,
	author = "Fr{\'e}d{\'e}ric Rochon and Zhou Zhang",
	title = {{Asymptotics of complete K{\"a}hler metrics of finite volume on quasiprojective manifolds.}},
	journal = "Adv. Math. ",
	volume = "231",
	number = "5",
	pages = "2892--2952",
	year = "2012",
	doi = "10.1016/j.aim.2012.08.005",
	classmath = {{*53C55 (Complex differential geometry (global)) 32Q20 (K{\"a}hler-Einstein manifolds) 53C44 (Geometric evolution equations (mean curvature flow)) 14C20 (Divisors, linear systems, invertible sheaves) }}
}

@Article{ TZ,
	AUTHOR = "Gang Tian and Zhou Zhang",
	TITLE = {On the {K}{\"a}hler-{R}icci flow on projective manifolds of general type},
	JOURNAL = "Chinese Ann. Math. Ser. B",
	FJOURNAL = "Chinese Annals of Mathematics. Series B",
	VOLUME = "27",
	YEAR = "2006",
	NUMBER = "2",
	PAGES = "179--192",
	ISSN = "0252-9599",
	CODEN = "CHAMEB",
	MRCLASS = "32Q20 (14E30 53C44)",
	MRNUMBER = "2243679 (2007c:32029)",
	MRREVIEWER = "Julien Keller",
	DOI = "10.1007/s11401-005-0533-x",
	URL = "http://dx.doi.org/10.1007/s11401-005-0533-x"
}

@Article{ Zhang,
	AUTHOR = "Zhou Zhang",
	TITLE = {On degenerate {M}onge-{A}mp{\`e}re equations over closed {K}{\"a}hler manifolds},
	JOURNAL = "Int. Math. Res. Not.",
	FJOURNAL = "International Mathematics Research Notices",
	YEAR = "2006",
	PAGES = "Art. ID 63640, 18",
	ISSN = "1073-7928",
	MRCLASS = "32W20 (32Q15)",
	MRNUMBER = "2233716 (2007b:32058)",
	MRREVIEWER = "S{\l}awomir Ko{\l}odziej",
	DOI = "10.1155/IMRN/2006/63640",
	URL = "http://dx.doi.org/10.1155/IMRN/2006/63640"
}

@Article{ Bou02,
	author = "S{\'e}bastien Boucksom",
	title = "{On the volume of a line bundle.}",
	journal = "Int. J. Math. ",
	volume = "13",
	number = "10",
	pages = "1043--1063",
	year = "2002",
	doi = "10.1142/S0129167X02001575",
	reviewer = "{Eberhard Oeljeklaus (Bremen)}",
	keywords = {{volume, compact K{\"a}hler manifold, close positive current, pseudoeffective class}},
	classmath = {{*14C20 (Divisors, linear systems, invertible sheaves) 32J25 (Transcendental methods of algebraic geometry) 32J27 (Compact K{\"a}hler manifolds) }}
}

@Article{ El2,
	AUTHOR = "R. Elkik",
	TITLE = "M{\'e}triques sur les fibr{\'e}s d'intersection",
	JOURNAL = "Duke Math. J.",
	FJOURNAL = "Duke Mathematical Journal",
	VOLUME = "61",
	YEAR = "1990",
	NUMBER = "1",
	PAGES = "303--328",
	ISSN = "0012-7094",
	CODEN = "DUMJAO",
	MRCLASS = "14G40 (14C17 57R20 58G26)",
	MRNUMBER = "1068389 (92b:14012)",
	MRREVIEWER = "Bruce Hunt",
	DOI = "10.1215/S0012-7094-90-06113-7",
	URL = "http://dx.doi.org/10.1215/S0012-7094-90-06113-7"
}

@Article{ Schum,
	hyphenation = "american",
	author = "Georg Schumacher",
	title = "Positivity of relative canonical bundles and applications",
	date = "2012-01-13",
	year = "2012",
	eprinttype = "arxiv",
	archivePrefix = "arXiv",
	journal = "arXiv:1201.2930"
}

@Article{ OSS,
	hyphenation = "american",
	author = "Yuji Odaka and Cristiano Spotti and Song Sun",
	title = {Compact Moduli Spaces of Del Pezzo Surfaces and K{\"a}hler-Einstein metrics},
	date = "2012-11-23",
	year = "2012",
	eprinttype = "arxiv",
	archivePrefix = "arXiv",
	joural = "arXiv:1210.0858"
}

@Book{ Har,
	AUTHOR = "Robin Hartshorne",
	TITLE = "Algebraic geometry",
	NOTE = "Graduate Texts in Mathematics, No. 52",
	PUBLISHER = "Springer-Verlag",
	ADDRESS = "New York",
	YEAR = "1977",
	PAGES = "xvi+496",
	ISBN = "0-387-90244-9",
	MRCLASS = "14-01",
	MRNUMBER = "0463157 (57 \#3116)",
	MRREVIEWER = "Robert Speiser"
}

@Article{ Karu,
	AUTHOR = "Kalle Karu",
	TITLE = "Minimal models and boundedness of stable varieties",
	JOURNAL = "J. Algebraic Geom.",
	FJOURNAL = "Journal of Algebraic Geometry",
	VOLUME = "9",
	YEAR = "2000",
	NUMBER = "1",
	PAGES = "93--109",
	ISSN = "1056-3911",
	MRCLASS = "14J10 (14D20 14E30 14J17)",
	MRNUMBER = "1713521 (2001g:14059)",
	MRREVIEWER = "Alessio Corti"
}

@Article{ KSB,
	AUTHOR = "J. Koll{\'a}r and N. I. Shepherd-Barron",
	TITLE = "Threefolds and deformations of surface singularities",
	JOURNAL = "Invent. Math.",
	FJOURNAL = "Inventiones Mathematicae",
	VOLUME = "91",
	YEAR = "1988",
	NUMBER = "2",
	PAGES = "299--338",
	ISSN = "0020-9910",
	CODEN = "INVMBH",
	MRCLASS = "14J10 (14D20 14J30 32G10 32G13)",
	MRNUMBER = "922803 (88m:14022)",
	MRREVIEWER = "Yujiro Kawamata",
	DOI = "10.1007/BF01389370",
	URL = "http://dx.doi.org/10.1007/BF01389370"
}

@Article{ Tra,
	AUTHOR = "Carlo Traverso",
	TITLE = "Seminormality and {P}icard group",
	JOURNAL = "Ann. Scuola Norm. Sup. Pisa (3)",
	VOLUME = "24",
	YEAR = "1970",
	PAGES = "585--595",
	MRCLASS = "14.55 (16.00)",
	MRNUMBER = "0277542 (43 \#3275)",
	MRREVIEWER = "P. Abellanas"
}

@Article{ GT,
	AUTHOR = "S. Greco and C. Traverso",
	TITLE = "On seminormal schemes",
	JOURNAL = "Compositio Math.",
	FJOURNAL = "Compositio Mathematica",
	VOLUME = "40",
	YEAR = "1980",
	NUMBER = "3",
	PAGES = "325--365",
	ISSN = "0010-437X",
	CODEN = "CMPMAF",
	MRCLASS = "14M05 (14B99)",
	MRNUMBER = "571055 (81j:14030)",
	MRREVIEWER = "S. L. Kleiman",
	URL = "http://www.numdam.org/item?id=CM_1980__40_3_325_0"
}

@Book{ Vie,
	AUTHOR = "Eckart Viehweg",
	TITLE = "Quasi-projective moduli for polarized manifolds",
	SERIES = "Ergebnisse der Mathematik und ihrer Grenzgebiete (3) [Results in Mathematics and Related Areas (3)]",
	VOLUME = "30",
	PUBLISHER = "Springer-Verlag",
	ADDRESS = "Berlin",
	YEAR = "1995",
	PAGES = "viii+320",
	ISBN = "3-540-59255-5",
	MRCLASS = "14-02 (14D20 14D22)",
	MRNUMBER = "1368632 (97j:14001)",
	MRREVIEWER = "P. E. Newstead"
}

@Article{ KSS,
	AUTHOR = "S{\'a}ndor J. Kov{\'a}cs and Karl Schwede and Karen E. Smith",
	TITLE = "The canonical sheaf of {D}u {B}ois singularities",
	JOURNAL = "Adv. Math.",
	FJOURNAL = "Advances in Mathematics",
	VOLUME = "224",
	YEAR = "2010",
	NUMBER = "4",
	PAGES = "1618--1640",
	ISSN = "0001-8708",
	CODEN = "ADMTA4",
	MRCLASS = "14J17 (14B05 14E15 32C35 32L20 32S20)",
	MRNUMBER = "2646306 (2011m:14062)",
	MRREVIEWER = "Tom{\'a}s S{\'a}nchez-Giralda",
	DOI = "10.1016/j.aim.2010.01.020",
	URL = "http://dx.doi.org/10.1016/j.aim.2010.01.020"
}

@Article{ fr,
	author = "Gerard {Freixas i Montplet}",
	title = "{An arithmetic Hilbert-Samuel theorem for pointed stable curves.}",
	journal = "J. Eur. Math. Soc. (JEMS) ",
	volume = "14",
	number = "2",
	pages = "321--351",
	year = "2012",
	doi = "10.4171/JEMS/304",
	reviewer = "{Shun Tang (Bonn)}",
	keywords = "{Arakelov theory, pointed stable curve, Mumford isomorphism, hyperbolic metric, Quillen metric, Selberg zeta function}",
	classmath = "{*14G40 (Arithmetic varieties and schemes) }"
}

@Book{ Kollar,
	AUTHOR = "J{\'a}nos Koll{\'a}r",
	TITLE = "Book on Moduli of Surfaces ",
	NOTE = "ongoing project, avalaible at the author's webpage \url{https://web.math.princeton.edu/~kollar/book/chap3.pdf}"
}

@Article{ Dem85,
	AUTHOR = "Jean-Pierre Demailly",
	TITLE = "Mesures de {M}onge-{A}mp{\`e}re et caract{\'e}risation g{\'e}om{\'e}trique des vari{\'e}t{\'e}s alg{\'e}briques affines",
	JOURNAL = "M{\'e}m. Soc. Math. France (N.S.)",
	FJOURNAL = "M{\'e}moires de la Soci{\'e}t{\'e} Math{\'e}matique de France. Nouvelle S{\'e}rie",
	NUMBER = "19",
	YEAR = "1985",
	PAGES = "124",
	ISSN = "0037-9484",
	MRCLASS = "32H35 (32C10 32F05)",
	MRNUMBER = "813252 (87g:32030)",
	MRREVIEWER = "G. M. Khenkin"
}

@Article{ PS,
	AUTHOR = "D. H. Phong and Jacob Sturm",
	TITLE = "Scalar curvature, moment maps, and the {D}eligne pairing",
	JOURNAL = "Amer. J. Math.",
	FJOURNAL = "American Journal of Mathematics",
	VOLUME = "126",
	YEAR = "2004",
	NUMBER = "3",
	PAGES = "693--712",
	ISSN = "0002-9327",
	CODEN = "AJMAAN",
	MRCLASS = "53D20 (32J27 53C21)",
	MRNUMBER = "2058389 (2005b:53137)",
	MRREVIEWER = "Michael J. Usher",
	URL = "http://muse.jhu.edu/journals/american_journal_of_mathematics/v126/126.3phong.pdf"
}

@Article{ Kov,
	author = "S{\'a}ndor J. Kov{\'a}cs",
	title = "Singularities of stable varieties",
	date = "2012-01-23",
	year = "2012",
	journal = "to appear in Handbook of Moduli, arXiv:1102.1240",
	eprinttype = "arxiv",
	archivePrefix = "arXiv",
	eprint = "1102.1240"
}

@Article{ GR,
	author = "Hans Grauert and Reinhold Remmert",
	title = {{Plurisubharmonische Funktionen in komplexen R{\"a}umen.}},
	journal = "Math. Z. ",
	volume = "65",
	pages = "175--194",
	year = "1956",
	doi = "10.1007/BF01473877",
	keywords = "{complex functions}"
}

@Book{ Deb,
	AUTHOR = "Olivier Debarre",
	TITLE = "Higher-dimensional algebraic geometry",
	SERIES = "Universitext",
	PUBLISHER = "Springer-Verlag",
	ADDRESS = "New York",
	YEAR = "2001",
	PAGES = "xiv+233",
	ISBN = "0-387-95227-6",
	MRCLASS = "14-02 (14E30 14Jxx)",
	MRNUMBER = "1841091 (2002g:14001)",
	MRREVIEWER = "Mark Gross"
}

@Article{ BBP,
	author = "S. Boucksom and A. Broustet and G. Pacienza",
	title = "{Uniruledness of stable base loci of adjoint linear systems with and without Mori Theory}",
	journal = "to appear in Math. Zeit., arXiv:0902.1142",
	year = "2010"
}

@Book{ KM,
	AUTHOR = "J{\'a}nos Koll{\'a}r and Shigefumi Mori",
	TITLE = "Birational geometry of algebraic varieties",
	SERIES = "Cambridge Tracts in Mathematics",
	VOLUME = "134",
	NOTE = "With the collaboration of C. H. Clemens and A. Corti, Translated from the 1998 Japanese original",
	PUBLISHER = "Cambridge University Press",
	ADDRESS = "Cambridge",
	YEAR = "1998",
	PAGES = "viii+254",
	ISBN = "0-521-63277-3",
	MRCLASS = "14E30",
	MRNUMBER = "1658959 (2000b:14018)",
	MRREVIEWER = "Mark Gross",
	DOI = "10.1017/CBO9780511662560",
	URL = "http://dx.doi.org/10.1017/CBO9780511662560"
}

@Article{ Ber2,
	AUTHOR = "Robert J. Berman",
	TITLE = "Bergman kernels and equilibrium measures for line bundles over projective manifolds",
	JOURNAL = "Amer. J. Math.",
	VOLUME = "131",
	YEAR = "2009",
	NUMBER = "5",
	PAGES = "1485--1524",
	FJOURNAL = "American Journal of Mathematics"
}

@InCollection{ BD,
	AUTHOR = "Robert J. Berman and Jean-Pierre Demailly",
	TITLE = "Regularity of plurisubharmonic upper envelopes in big cohomology classes",
	BOOKTITLE = "Perspectives in analysis, geometry, and topology",
	SERIES = "Progr. Math.",
	VOLUME = "296",
	PAGES = "39--66",
	PUBLISHER = {Birkh{\"a}user/Springer, New York},
	YEAR = "2012"
}

@Article{ Sko,
	AUTHOR = "Henri Skoda",
	TITLE = "Sous-ensembles analytiques d'ordre fini ou infini dans {${\bf C}^{n}$}",
	JOURNAL = "Bull. Soc. Math. France",
	FJOURNAL = "Bulletin de la Soci{\'e}t{\'e} Math{\'e}matique de France",
	VOLUME = "100",
	YEAR = "1972",
	PAGES = "353--408"
}

@Article{ Maz,
	author = "R. Mazzeo",
	title = {{K{\"a}hler-Einstein metrics singular along a smooth divisor}},
	journal = {Journ{\'e}es "{\'E}quations aux d{\'e}riv{\'e}es partielles" (Saint Jean-de-Mont, 1999)},
	year = 1999
}

@Article{ UY,
	author = "K. Uhlenbeck and S.T. Yau",
	title = "{On the existence of Hermitian-Yang-Mills connections in stable vector bundles.}",
	journal = "Commun. Pure Appl. Math.",
	volume = "39",
	pages = "S257--S293",
	year = "1986",
	doi = "10.1002/cpa.3160390714"
}

@Article{ Don85,
	author = "S.K. Donaldson",
	title = "{Anti self-dual Yang Mills connections over complex algebraic surfaces and stable vector bundles}",
	journal = "Proc. Lond. Math. Soc., III. Ser.",
	volume = "50",
	pages = "1--26",
	year = "1985"
}

@Article{ GZ07,
	author = "Vincent Guedj and Ahmed Zeriahi",
	title = "{The weighted Monge-Amp{\`e}re energy of quasi plurisubharmonic functions}",
	journal = "J. Funct. An.",
	year = "2007",
	volume = "250",
	pages = "442--482"
}

@Article{ D1,
	AUTHOR = "Jean-Pierre Demailly",
	TITLE = {Estimations {$L^{2}$} pour l'op{\'e}rateur {$\bar \partial $} d'un fibr{\'e} vectoriel holomorphe semi-positif au-dessus d'une vari{\'e}t{\'e} k{\"a}hl{\'e}rienne compl{\`e}te},
	JOURNAL = "Ann. Sci. {\'E}cole Norm. Sup. (4)",
	FJOURNAL = "Annales Scientifiques de l'{\'E}cole Normale Sup{\'e}rieure. Quatri{\`e}me S{\'e}rie",
	VOLUME = "15",
	YEAR = "1982",
	NUMBER = "3",
	PAGES = "457--511",
	ISSN = "0012-9593",
	CODEN = "ASENAH",
	MRCLASS = "32L15 (32L20)",
	MRNUMBER = "690650 (85d:32057)",
	MRREVIEWER = "Autorreferat",
	URL = "http://www.numdam.org/item?id=ASENS_1982_4_15_3_457_0"
}

@Article{ D2,
	AUTHOR = "Jean-Pierre Demailly",
	TITLE = "Regularization of closed positive currents and intersection theory",
	JOURNAL = "J. Algebraic Geom.",
	FJOURNAL = "Journal of Algebraic Geometry",
	VOLUME = "1",
	YEAR = "1992",
	NUMBER = "3",
	PAGES = "361--409",
	ISSN = "1056-3911",
	MRCLASS = "32C30 (32C17 32J25)",
	MRNUMBER = "1158622 (93e:32015)",
	MRREVIEWER = "Takeo Ohsawa"
}

@Article{ MR,
	AUTHOR = "Rafe Mazzeo and Yanir A. Rubinstein",
	TITLE = {The {R}icci continuity method for the complex {M}onge--{A}mp{\`e}re equation, with applications to {K}{\"a}hler--{E}instein edge metrics},
	JOURNAL = "C. R. Math. Acad. Sci. Paris",
	VOLUME = "350",
	YEAR = "2012",
	NUMBER = "13-14",
	PAGES = "693--697",
	ISSN = "1631-073X",
	DOI = "10.1016/j.crma.2012.07.001",
	URL = "http://dx.doi.org/10.1016/j.crma.2012.07.001",
	FJOURNAL = "Comptes Rendus Math{\'e}matique. Acad{\'e}mie des Sciences. Paris",
	MRCLASS = "Preliminary Data",
	MRNUMBER = "2971382"
}

@Article{ Tian,
	author = "G. Tian",
	title = {{On K{\"a}hler-Einstein metrics on certain K{\"a}hler manifolds with $c\_1(M)>0$}},
	journal = "Invent. Math",
	year = "1987",
	volume = "89",
	pages = "225--246"
}


@book {HL,
    AUTHOR = {Han, Qing and Lin, Fanghua},
     TITLE = {Elliptic partial differential equations},
    SERIES = {Courant Lecture Notes in Mathematics},
    VOLUME = {1},
 PUBLISHER = {New York University Courant Institute of Mathematical
              Sciences},
   ADDRESS = {New York},
      YEAR = {1997},
     PAGES = {x+144},
      ISBN = {0-9658703-0-8; 0-8218-2691-3},
   MRCLASS = {35Jxx (35-01 35B50)},
  MRNUMBER = {1669352 (2001d:35035)},
}

@Article{ Jef,
	author = "T. Jeffres",
	title = {{Uniqueness of K{\"a}hler-Einstein cone metrics}},
	journal = "Publ. Mat. 44",
	year = "2000",
	volume = "44",
	number = "2",
	pages = "437--448"
}

@Article{ Paun,
	author = "Mihai P{\u{a}}un",
	title = {{Regularity properties of the degenerate Monge-Amp{\`e}re equations on compact K{\"a}hler manifolds.}},
	journal = "Chin. Ann. Math., Ser. B",
	year = 2008,
	volume = "29",
	number = 6,
	pages = "623--630"
}

@Article{ rber,
	author = "Robert J. Berman",
	title = {{A thermodynamical formalism for Monge-Amp{\`e}re equations, Moser-Trudinger inequalities and K{\"a}hler-Einstein metrics}},
	journal = "arXiv:1011.3976",
	year = "2011"
}

@Article{ bobot,
	author = "B. Berndtsson",
	title = "{$L^2$-extension of $\dbar$-closed forms}",
	journal = "arXiv:1104.4620",
	year = "2011",
	date-added = "2011-05-03 22:56:39 +0200",
	date-modified = "2011-05-03 23:00:27 +0200"
}

@Article{ Zha,
	author = "Z. Zhang",
	title = {{On degenerate Monge-Amp{\`e}re equations over closed K{\"a}hler manifolds}},
	journal = "Int. Math. Res. Let.",
	year = "2006",
	date-added = "2011-05-02 23:03:49 +0200",
	date-modified = "2011-05-02 23:40:10 +0200"
}

@Article{ ST,
	AUTHOR = "G{\'a}bor Sz{\'e}kelyhidi and Valentino Tosatti",
	TITLE = "Regularity of weak solutions of a complex {M}onge-{A}mp{\`e}re equation",
	JOURNAL = "Anal. PDE",
	FJOURNAL = "Analysis \& PDE",
	VOLUME = "4",
	YEAR = "2011",
	NUMBER = "3",
	PAGES = "369--378",
	ISSN = "1948-206X",
	MRCLASS = "32W20 (35J96 35K55 53C44)",
	MRNUMBER = "2872120 (2012k:32056)",
	MRREVIEWER = "S{\l}awomir Ko{\l}odziej",
	DOI = "10.2140/apde.2011.4.369",
	URL = "http://dx.doi.org/10.2140/apde.2011.4.369"
}

@Article{ Ev,
	author = "L.C Evans",
	title = "{Classical solutions of fully nonlinear, convex, second order elliptic equations }",
	journal = "Comm. Pure Appl. Math.",
	year = "1982",
	volume = "35",
	number = "3",
	pages = "333--363",
	date-added = "2011-05-02 23:01:06 +0200",
	date-modified = "2011-05-02 23:38:33 +0200"
}

@Article{ tos2,
	author = "V. Tosatti",
	title = {{Adiabatic limits of Ricci-flat K{\"a}hler metrics}},
	journal = "J. Diff. Geom.",
	year = "2010",
	volume = "84",
	date-added = "2011-05-02 23:01:06 +0200",
	date-modified = "2011-05-02 23:38:33 +0200"
}

@Article{ tos1,
	author = "V. Tosatti",
	title = {{Limits of Calabi-Yau metrics when the K{\"a}hler class degenerates}},
	journal = "JEMS",
	year = "2009",
	date-added = "2011-05-02 22:58:34 +0200",
	date-modified = "2011-05-02 23:39:53 +0200"
}

@Article{ DemPal,
	author = "Jean-Pierre Demailly and N. Pali",
	title = {{Degenerate complex Monge-Amp{\`e}re equations over compact K{\"a}hler manifolds}},
	journal = "Internat. J. Math.",
	year = "2010",
	volume = "21",
	date-added = "2011-05-02 22:56:17 +0200",
	date-modified = "2011-05-02 23:39:08 +0200"
}

@Article{ Sug,
	author = "K. Sugiyama",
	title = {{Einstein-K{\"a}hler metrics on minimal varieties of general type and an inequality between Chern numbers}},
	journal = "Adv. Stud. Pure Math.",
	year = "1990",
	date-added = "2011-05-02 22:53:14 +0200",
	date-modified = "2011-05-02 22:56:00 +0200"
}

@Article{ SonTia,
	author = "J. Song and G Tian",
	title = {{Canonical measures and K{\"a}hler-Ricci flow}},
	journal = "arXiv: 08022570",
	year = "2008",
	date-added = "2011-05-02 22:50:38 +0200",
	date-modified = "2011-05-02 23:38:49 +0200"
}

@Article{ Kolo,
	author = "S. Ko{\l}odziej",
	title = "{The complex Monge-Amp{\`e}re operator}",
	journal = "Acta Math.",
	year = "1998",
	volume = "180",
	number = "1",
	pages = "69--117",
	date-added = "2011-05-02 22:48:49 +0200",
	date-modified = "2011-05-03 23:06:00 +0200"
}

@Article{ Kolo2,
	title = {{H{\"o}lder continuity of solutions to the complex Monge-Amp{\`e}re equation with the right-hand side in $L^p$: the case of compact K{\"a}hler manifolds}},
	author = "S. Ko{\l}odziej",
	journal = "Math. Ann.",
	pages = "379--386",
	volume = "342",
	number = "1",
	year = "2008",
	date-added = "2011-05-02 22:48:49 +0200",
	date-modified = "2011-05-03 23:06:00 +0200"
}

@Article{ EGZ,
	author = "Ph. Eyssidieux and V. Guedj and A. Zeriahi",
	title = {{Singular K{\"a}hler-Einstein metrics}},
	journal = "{J. Amer. Math. Soc.}",
	year = "2009",
	pages = "607--639",
	volume = "22",
	date-added = "2011-05-02 22:46:49 +0200",
	date-modified = "2011-05-02 22:48:44 +0200"
}

@Article{ BBEGZ,
	author = "Robert J. Berman and S. Boucksom and Ph. Eyssidieux and V. Guedj and A. Zeriahi",
	title = {{K{\"a}hler-Einstein metrics and the K{\"a}hler-Ricci flow on log-Fano varieties}},
	journal = "arXiv:1111.7158v2",
	year = "2011",
	date-added = "2011-05-02 22:46:49 +0200",
	date-modified = "2011-05-02 22:48:44 +0200"
}

@Article{ BT,
	author = "E. Bedford and B.A. Taylor",
	title = "{Fine topology, Silov boundary, and $(dd^c)^n$}",
	journal = "J. Funct. Anal.",
	year = "1987",
	volume = "72",
	number = "2",
	pages = "225--251"
}

@Article{ BanKob,
	author = "S. Bando and R. Kobayashi",
	title = {{Ricci flat K{\"a}hler metrics on affine algebraic manifolds}},
	journal = "Math. Ann.",
	year = "1990",
	volume = "287",
	pages = "175--180",
	date-added = "2011-05-02 22:42:32 +0200",
	date-modified = "2011-05-02 23:37:13 +0200"
}

@InCollection{ Chern,
	title = "{On holomorphic mappings of hermitian manifolds of the same dimension.}",
	author = "Shiing-Shen Chern",
	publisher = "{}",
	year = "1968",
	keywords = "{complex functions}"
}

@Article{ KobR,
	author = "R. Kobayashi",
	title = {{K{\"a}hler-Einstein metric on an open algebraic manifolds}},
	journal = "Osaka 1. Math.",
	year = "1984",
	volume = "21",
	pages = "399--418",
	date-added = "2011-05-02 22:36:40 +0200",
	date-modified = "2011-05-02 23:39:34 +0200"
}

@article{DDGHKZ,

author = {Jean-Pierre Demailly and Slawomir Dinew and Vincent Guedj and Hoang Hiep Pham and Slawomir Kolodziej and Ahmed Zeriahi},
title = {H{\"o}lder continuous solutions to Monge-Amp{\`e}re equations},
date = {2011-12-06},
year = {2011},
eprinttype = {arxiv},
journal = {arXiv:1112.1388}
} 

@Book{ Gilb,
	author = "D. Gilbarg and N. Trudinger",
	title = "Elliptic partial differential equations of second order",
	publisher = "Springer-Verlag",
	year = "1977",
	date-added = "2011-05-02 22:34:09 +0200",
	date-modified = "2011-05-02 23:26:00 +0200"
}

@Article{ Kry,
	author = "N. V. Krylov",
	title = "Smoothness of the payoff function for a controllable diffusion process in a domain. (Russian)",
	journal = "Izv. Akad. Nauk SSSR Ser. Mat.",
	year = "1989",
	volume = "53",
	number = "1",
	pages = "399--418",
	note = "translation in Math. USSR-Izv. 34 (1990), no. 1, 65--95",
	pages1 = "66--96"
}

@Article{ BEGZ,
	author = "S{\'e}bastien Boucksom and Philippe Eyssidieux and Vincent Guedj and Ahmed Zeriahi",
	title = "{Monge-Amp{\`e}re equations in big cohomology classes.}",
	journal = "Acta Math.",
	volume = "205",
	number = "2",
	pages = "199--262",
	year = "2010",
	doi = "10.1007/s11511-010-0054-7"
}

@book {GHL,
    AUTHOR = {Gallot, Sylvestre and Hulin, Dominique and Lafontaine,
              Jacques},
     TITLE = {Riemannian geometry},
    SERIES = {Universitext},
   EDITION = {Third},
 PUBLISHER = {Springer-Verlag},
   ADDRESS = {Berlin},
      YEAR = {2004},
     PAGES = {xvi+322},
      ISBN = {3-540-20493-8},
   MRCLASS = {53-01 (53C20 53C21 53C23)},
  MRNUMBER = {2088027 (2005e:53001)},
MRREVIEWER = {Joseph E. Borzellino},
       DOI = {10.1007/978-3-642-18855-8},
       URL = {http://dx.doi.org/10.1007/978-3-642-18855-8},
}


@Article{ Clodo,
	author = "Beno{\^i}t Claudon",
	title = "{$\Gamma$-reduction for smooth orbifolds}",
	journal = "Manuscr. Math.",
	volume = "127",
	number = "4",
	pages = "521--532",
	year = "2008",
	doi = "10.1007/s00229-008-0215-6"
}

@Article{ Camp2,
	author = "Fr{\'e}d{\'e}ric Campana",
	title = {{Orbifoldes g{\'e}om{\'e}triques sp{\'e}ciales et classification bim{\'e}romorphe des vari{\'e}t{\'e}s K{\"a}hl{\'e}riennes compactes.}},
	journal = "J. Inst. Math. Jussieu",
	volume = "10",
	number = "4",
	pages = "809--934",
	year = "2011"
}

@InCollection{ Camp1,
	author = "Fr{\'e}d{\'e}ric Campana",
	title = "{Special orbifolds and birational classification: a survey.}",
	publisher = {{Z{\"u}rich: European Mathematical Society (EMS)}},
	year = "2011"
}

@Article{ Brendle,
	author = "S. Brendle",
	title = {{Ricci flat K{\"a}hler metrics with edge singularities }},
	journal = "to appear in IMRN, arXiv 1103.5454",
	year = "2011",
	date-added = "2011-04-21 13:52:23 +0200",
	date-modified = "2011-04-21 13:53:19 +0200"
}

@InCollection{ Don,
	AUTHOR = "S. K. Donaldson",
	TITLE = {K{\"a}hler metrics with cone singularities along a divisor},
	BOOKTITLE = "Essays in mathematics and its applications",
	PAGES = "49--79",
	PUBLISHER = "Springer",
	ADDRESS = "Heidelberg",
	YEAR = "2012",
	MRCLASS = "32Q20",
	MRNUMBER = "2975584",
	DOI = "10.1007/978-3-642-28821-0\_4",
	URL = "http://dx.doi.org/10.1007/978-3-642-28821-0_4"
}

@Article{ Aubin,
	author = "T. Aubin",
	title = {{{\'E}quations du type Monge-Amp{\`e}re sur les vari{\'e}t{\'e}s K{\"a}hl{\'e}riennes compactes}},
	journal = "Bull. Sc. Math.",
	year = "1978",
	volume = "102",
	date-added = "2011-04-21 13:48:41 +0200",
	date-modified = "2011-04-21 13:50:37 +0200"
}

@Article{ Yau78,
	author = "Shing-Tung Yau",
	title = {{On the Ricci curvature of a compact K{\"a}hler manifold and the complex Monge-Amp{\`e}re equation. I.}},
	journal = "Commun. Pure Appl. Math.",
	volume = "31",
	pages = "339--411",
	year = "1978"
}

@Article{ Tsuji07,
	author = "H. Tsuji",
	title = {{Dynamical construction if K{\"a}hler-Einstein metrics}},
	journal = "arXiv:0606626",
	year = "2007"
}

@Article{ BCHM,
	author = "C. Birkar and P. Cascini and C. Hacon and J. McKernan",
	title = "{Existence of minimal models for varieties of log general type}",
	journal = "J. Amer. Math. Soc.",
	year = "2010",
	volume = "23",
	pages = "405--468"
}

@Article{ Dem95,
	author = "Jean-Pierre Demailly",
	title = "{Algebraic criteria for Kobayashi hyperbolic projective varieties and jet differentials}",
	journal = "Proceedings of the Symposia in Pure Maths.,",
	year = "1995",
	number = "62.2",
	address = "AMS Summer Institute on Algebraic Geometry, Santa Cruz",
	date-added = "2011-04-21 13:39:32 +0200",
	date-modified = "2011-05-02 23:16:57 +0200"
}

@Article{ DPS,
	author = "Jean-Pierre Demailly and T. Peternell and M. Schneider",
	title = {{K{\"a}hler manifolds with numerically effective Ricci class}},
	journal = "Comp. Math.",
	date-added = "2011-04-21 13:36:20 +0200",
	date-modified = "2011-04-21 13:39:14 +0200"
}

@Article{ Peternell,
	author = "Thomas Peternell",
	title = "{Varieties with generically nef tangent bundles.}",
	language = "English",
	journal = "J. Eur. Math. Soc. (JEMS)",
	volume = "14",
	number = "2",
	pages = "571--603",
	year = "2012",
	doi = "10.4171/JEMS/312"
}

@Book{ BY,
	author = "S. Bochner and K. Yano",
	title = "{Curvature and Betti numbers}",
	publisher = "Princeton University Press",
	year = "1953",
	number = "32",
	series = "Annals of Mathematical Studies",
	date-added = "2011-04-21 13:31:07 +0200",
	date-modified = "2011-04-21 13:34:48 +0200"
}

@Article{ Li71,
	author = "A Lichnerowicz",
	title = {{Vari{\'e}t{\'e}s k{\"a}hleriennes {\`a} premi{\`e}re classe de Chern non n{\'e}gative et vari{\'e}t{\'e}s riemanniennes {\`a} courbure de Ricci g{\'e}n{\'e}ralis{\'e}e non n{\'e}gative}},
	journal = "J. Diff. Geom.",
	year = "1971",
	number = "6",
	date-added = "2011-04-21 13:29:53 +0200",
	date-modified = "2011-04-21 13:30:44 +0200"
}

@Article{ Li67,
	author = "A Lichnerowicz",
	title = {{Vari{\'e}t{\'e}s k{\"a}hleriennes et premi{\`e}re classe de Chern}},
	journal = "J. Diff. Geom.",
	year = "1967",
	number = "1",
	date-added = "2011-04-21 13:28:09 +0200",
	date-modified = "2011-04-21 13:29:51 +0200"
}

@Article{ Kob81,
	author = "S. Kobayashi",
	title = "Recent results in complex differential geometry",
	journal = "Jber. Math.-Verein.",
	year = "1981",
	volume = "83",
	pages = "147--158",
	date-added = "2011-04-21 13:25:32 +0200",
	date-modified = "2011-04-21 13:27:03 +0200"
}

@Article{ Kob80,
	author = "S. Kobayashi",
	title = "{The first Chern class and holomorphic tensor fields}",
	journal = "Nagoya Math.",
	year = "1980",
	volume = "77",
	pages = "5--11",
	date-added = "2011-04-21 13:22:46 +0200",
	date-modified = "2011-04-21 13:25:24 +0200"
}

@Unpublished{ Dem1,
	author = "Jean-Pierre Demailly",
	title = "{Complex Analytic and Algebraic Geometry}",
	note = "Book in preparation"
}

@Article{ DK,
	author = "Jean-Pierre Demailly and J{\'a}nos Koll{\'a}r",
	title = {{Semi-continuity of complex singularity exponents and K{\"a}hler-Einstein metrics on Fano orbifolds}},
	journal = "{Ann. Sci. {\'E}cole Norm. Sup.}",
	year = "2001",
	volume = "34",
	pages = "525--556"
}

@Article{ DEL,
	author = "Jean-Pierre Demailly and L. Ein and R. Lazarsfeld",
	title = "{A subadditivity property of multiplier ideals}",
	journal = "{Michigan Math. J.}",
	year = "2000",
	volume = "48",
	pages = "137--156"
}

@Article{ Tak,
	author = "S. Takagi",
	title = "{Adjoint ideals along closed subvarieties of higher codimension}",
	journal = "{arXiv:0711.2342}",
	year = "2007"
}

@Article{ CGP,
	author = "Fr{\'e}d{\'e}ric Campana and Henri Guenancia and Mihai P\u{a}un",
	title = "{Metrics with cone singularities along normal crossing divisors and holomorphic tensor fields}",
	journal = "Ann. Scient. Éc. Norm. Sup.",
	year = "2013",
	volume = "46",
	pages = "879-916"
}

@article{GP,
hyphenation = {american},
author = {Henri Guenancia and Mihai P{\u{a}}un},
title = {Conic singularities metrics with prescribed Ricci curvature: the case of general cone angles along normal crossing divisors},
date = {2013-07-28},
year = {2013},
eprinttype = {arxiv},
journal = {arXiv:1307.6375}
} 

@Article{ Bl,
	author = "Zbigniew B{\l}ocki",
	title = "{The Calabi-Yau Theorem}",
	journal = {{to appear in Lecture Notes in Mathematics as a part of the volume Complex Monge-Amp{\`e}re equations and geodesics in the space of K{\"a}hler metrics (ed. V. Guedj)}},
	year = "2011",
	note = "{http://gamma.im.uj.edu.pl/$\sim$blocki/publ}"
}

@Article{ MZ,
	author = "Jeffery D. McNeal and Yunus E. Zeytuncu",
	title = "Multiplier Ideals and Integral Closure of Monomial Ideals: An Analytic Approach",
	journal = "{ arXiv:1001.4983}",
	year = "2010"
}

@Article{ Tia,
	author = "Gang Tian and Shing-Tung Yau",
	title = {{Existence of K{\"a}hler-Einstein metrics on complete K{\"a}hler manifolds and their applications to algebraic geometry}},
	journal = "Adv. Ser. Math. Phys. 1",
	year = "1987",
	volume = "1",
	pages = "574--628",
	note = "Mathematical aspects of string theory (San Diego, Calif., 1986)",
	date-modified = "2011-05-03 23:04:12 +0200",
	publisher = "World Sci. Publishing, Singapore"
}

@Book{ Siu,
	author = "Yum-Tong Siu",
	title = {{Lectures on Hermitian-Einstein Metrics for Stable Bundles and K{\"a}hler-Einstein Metrics}},
	publisher = {Birkh{\"a}user},
	year = "1987"
}

@Book{ Dem2,
	author = "Jean-Pierre Demailly and Jos{\'e} Bertin and Luc Illusie and Chris Peters",
	title = "{Introduction {\`a} la th{\'e}orie de Hodge}",
	publisher = "Soc. Math. de France",
	year = "1996",
	volume = "3",
	series = "{Panorama et Synth{\`e}ses}"
}

@Book{ Nak,
	author = "Noboru Nakayama",
	title = "{Zariski-decomposition and abundance}",
	publisher = "{MSJ Memoirs}",
	year = "2004",
	volume = "14"
}

@Book{ Dem3,
	author = "Jean-Pierre Demailly",
	title = "{Multiplier ideal sheaves and analytic methods in algebraic geometry}",
	publisher = "{ICTP}",
	year = "2001",
	volume = "6",
	series = "{Lecture Notes}"
}

@Book{ BEG,
	editor = "S{\'e}bastien Boucksom and Philippe Eyssidieux and Vincent Guedj",
	title = {{Introduction to the K{\"a}hler-Ricci flow}},
	publisher = "{Springer}",
	year = "2013",
	series = "{Lecture Notes in Mathematics, to appear}"
}

@Article{ FJ,
	author = "Charles Favre and Mattias Jonsson",
	title = "{Valuations and Multiplier Ideals}",
	journal = "{J. Amer. Math. Soc.}",
	year = "2005",
	volume = "18",
	pages = "655--684"
}

@Article{ How,
	author = "Jason Howald",
	title = "{Multiplier Ideals of Monomial Ideals}",
	journal = "{Trans. Amer. Math. Soc.}",
	year = "2001",
	volume = "353(7)",
	pages = "2665--2671"
}

@Book{ Rob,
	author = "A. Wayne Roberts and Dale E. Varberg",
	title = "{Convex Functions}",
	publisher = "Academic Press",
	year = "1973"
}

@Book{ AM,
	author = "Eric Amar and Etienne Matheron",
	title = "{Analyse Complexe}",
	publisher = "Cassini",
	year = "2004"
}

@Article{ GKKP,
	author = "Daniel Greb and Stefan Kebekus and S{\'a}ndor J. Kov{\'a}cs and Thomas Peternell",
	title = "Differential forms on log canonical spaces",
	journal = "Publ. Math. Inst. Hautes {\'E}tudes Sci.",
	pages = "87--169",
	year = "2011",
	doi = "10.1007/s10240-011-0036-0",
	number = "114",
	issn = "0073-8301",
	url = "http://dx.doi.org/10.1007/s10240-011-0036-0",
	FJOURNAL = "Publications Math{\'e}matiques. Institut de Hautes {\'E}tudes Scientifiques",
	MRCLASS = "14F10 (14C20 14E15 14F05)",
	MRNUMBER = "2854859",
	MRREVIEWER = "Anne-Sophie Kaloghiros"
}

@Misc{ bob,
	author = "Bo Berndtsson",
	title = "{$L^2$ methods for the $\db$-equation}",
	year = "1995",
	note = "Lecture Notes",
	institution = "Chalmers University of Technology"
}

@Book{ Hor,
	author = {Lars H{\"o}rmander},
	title = "{An Introduction to complex analysis in several variables}",
	publisher = "North-Holland Math. Libr.",
	year = "1973"
}

@Book{ Hor2,
	author = {Lars H{\"o}rmander},
	title = "{Notions of convexity}",
	publisher = {Birkh{\"a}user},
	year = "1994"
}

@Book{ Kli,
	author = "Maciej Klimek",
	title = "{Pluripotential Theory}",
	publisher = "Oxford Univ. Press",
	year = "1991"
}

@Article{ bob2,
	author = "Bo Berndtsson",
	title = "{Integral formulas and the Ohsawa-Takegoshi extension theorem}",
	journal = "{Science in China Series A : Mathematics.}",
	year = "2005",
	volume = "48",
	pages = "61--73",
	month = "December"
}

@Article{ bob3,
	author = "Bo Berndtsson",
	title = "{The extension theorem of Ohsawa-Takegoshi and the theorem of Donnelly-Fefferman}",
	journal = "{Ann. Inst. Fourier (Grenoble)}",
	year = "1996",
	pages = "1083--1094",
	month = "December"
}

@Article{ Eis,
	author = "Eug{\`e}ne Eisenstein",
	title = "{Generalizations of the restriction theorem for multiplier ideals}",
	journal = "{J. Amer. Math. Soc.}",
	year = "2010",
	pages = "arXiv:1001:2841"
}

@Article{ bfj,
	author = "S{\'e}bastien Boucksom and Charles Favre and Mattias Jonsson",
	title = "{Valuations and plurisubharmonic singularities}",
	journal = "{Publ. RIMS, Kyoto Univ.}",
	year = "2008",
	volume = "44",
	pages = "449--494"
}

@Article{ Nad,
	author = "Alan Nadel",
	title = {{Multiplier ideal sheaves and existence of K{\"a}hler-Einstein metrics of positive scalar curvature}},
	journal = "{Ann. of Math.}",
	year = "1990",
	volume = "132",
	pages = "549--596"
}

@Book{ Lel,
	author = "Pierre Lelong",
	title = "{Fonctions plurisousharmoniques et formes diff{\'e}rentielles positives}",
	publisher = "Dunod, Paris, Gordon $\&$ Breach, New York",
	year = "1968"
}

@Book{ Kra,
	author = "Steven Krantz",
	title = "{Function theory of several complex variables}",
	publisher = "Wadsworth $\&$ Brooks/Cole",
	year = "1992"
}

@InCollection{ Gri,
	author = "Phillip A. Griffiths",
	title = "Entire Holomorphic Mappings in one and several Complex Variables",
	publisher = "Princeton University Press",
	year = "1976",
	editor = "Annals of Mathematics Studies"
}

@Article{ Auvray,
	author = "Hugues Auvray",
	title = {The space of {Poincar{\'e} type K{\"a}hler} metrics on the complement of a divisor},
	journal = "arXiv:1109.3159",
	year = "2011"
}

@Article{ CG,
	title = "{A defect relation for equidimensional holomorphic mappings between algebraic varieties.}",
	author = "James Carlson and Phillip Griffiths",
	journal = "Ann. Math.",
	pages = "557--584",
	volume = "95",
	year = 1972,
	key = "2",
	doi = "10.2307/1970871",
	classmath = "{*32H25 (Picard type theorems and generalizations (analytic spaces)) 14B10 (Infinitesimal methods) 32C30 (Integration on analytic sets and spaces) 32F45 (Invariant metrics and pseudodistances) }"
}

@Article{ JMR,
	author = "Thalia Jeffres and Rafe Mazzeo and Yanir Rubinstein",
	title = {{K{\"a}hler-Einstein metrics with edge singularities}},
	journal = "to appear in Ann. of Math., arXiv:1105.5216",
	year = 2011,
	note = "with an appendix by C. Li and Y. Rubinstein"
}

@Article{ CY,
	title = {{On the existence of a complete K{\"a}hler metric on non-compact complex manifolds and the regularity of Fefferman's equation.}},
	author = "Shiu-Yuen Cheng and Shing-Tung Yau",
	journal = "Commun. Pure Appl. Math.",
	pages = "507--544",
	volume = "33",
	year = "1980"
}

@Article{ Yau,
	author = "Shing-Tung Yau",
	title = "{Harmonic functions on complete Riemannian manifolds.}",
	journal = "Commun. Pure Appl. Math. ",
	volume = "28",
	pages = "201--228",
	year = "1975"
}

@Article{ Yau2,
	author = "Shing-Tung Yau",
	title = {{A general Schwarz lemma for K{\"a}hler manifolds.}},
	journal = "Amer. J. Math. ",
	volume = "100",
	pages = "197--203",
	year = "1978",
	doi = "10.2307/2373880",
	keywords = {{negatively pinched, Riemann surface, Schwarz-Pick-Ahlfors lemma, complete K{\"a}hler manifold, Ricci curvature bounded from below, holomorphic bisectional curvature}},
	classmath = "{*53C55 (Complex differential geometry (global)) 32F45 (Invariant metrics and pseudodistances) }"
}

@Article{ BT82,
	title = "A new capacity for plurisubharmonic functions",
	author = "E. Bedford and B.A. Taylor",
	journal = "Acta Math.",
	pages = "1--40",
	volume = "149",
	number = "1-2",
	year = "1982"
}

@Article{ BGZ,
	author = "Slimane Benelkourchi and Vincent Guedj and Ahmed Zeriahi",
	title = "{A priori estimates for weak solutions of complex Monge-Amp{\`e}re equations}",
	year = "2008",
	journal = "Ann. Sc. Norm. Super. Pisa, Cl. Sci.",
	volume = "5"
}

@Article{ ELMNP,
	author = "L. Ein and R. Lazarsfeld and M. Mustata and M. Nakamaye and M. Popa",
	title = "{Asymptotic invariants of base loci}",
	year = "2006",
	journal = "Ann. Inst. Fourier (Grenoble)",
	volume = "56",
	number = "6",
	pages = "1701--1734"
}

@Article{ BBGZ,
	author = "Robert J. Berman and S{\'e}bastien Boucksom and Vincent Guedj and Ahmed Zeriahi",
	title = "{A variational approach to complex Monge-Amp{\`e}re equations}",
	year = "2009",
	journal = "to appear in Publ. IHES, arXiv:0907.4490"
}

@Article{ FN,
	AUTHOR = "John Erik Forn{\ae}ss and Raghavan Narasimhan",
	TITLE = "The {L}evi problem on complex spaces with singularities",
	JOURNAL = "Math. Ann.",
	FJOURNAL = "Mathematische Annalen",
	VOLUME = "248",
	YEAR = "1980",
	NUMBER = "1",
	PAGES = "47--72",
	ISSN = "0025-5831",
	CODEN = "MAANA3",
	MRCLASS = "32E10",
	MRNUMBER = "569410 (81f:32020)",
	MRREVIEWER = "Peter Pflug",
	DOI = "10.1007/BF01349254",
	URL = "http://dx.doi.org/10.1007/BF01349254"
}

@Article{ GZ05,
	author = "Vincent Guedj and Ahmed Zeriahi",
	title = {{Intrinsic capacities on compact K{\"a}hler manifolds.}},
	journal = "J. Geom. Anal. ",
	volume = "15",
	number = "4",
	pages = "607--639",
	year = "2005",
	doi = "10.1007/BF02922247"
}

@Unpublished{ DemP,
	author = "Jean-Pierre Demailly",
	title = "Potential theory in several complex variables",
	note = "Lecture given at the CIMPA in 1989, completed by a conference given in Trento, 1992; avalaible at the author's webpage: \url{http://www-fourier.ujf-grenoble.fr/~demailly/books.html}"
}

@Article{ Kol,
	author = "S. Ko{\l}odziej",
	title = {{Stability of solutions to the complex Monge-Amp{\`e}re equations on compact K{\"a}hler manifolds}},
	year = "2001",
	journal = "Preprint"
}

@Article{ Tsuji,
	AUTHOR = "Hajime Tsuji",
	TITLE = "A characterization of ball quotients with smooth boundary",
	JOURNAL = "Duke Math. J.",
	FJOURNAL = "Duke Mathematical Journal",
	VOLUME = "57",
	YEAR = "1988",
	NUMBER = "2",
	PAGES = "537--553",
	ISSN = "0012-7094",
	CODEN = "DUMJAO",
	MRCLASS = "32J25 (14J40 53C55)",
	MRNUMBER = "962519 (89k:32054)",
	MRREVIEWER = "Koichi Ogiue",
	DOI = "10.1215/S0012-7094-88-05724-9",
	URL = "http://dx.doi.org/10.1215/S0012-7094-88-05724-9"
}

@Article{ Tsuji88,
	AUTHOR = "Hajime Tsuji",
	TITLE = {Existence and degeneration of {K}{\"a}hler-{E}instein metrics on minimal algebraic varieties of general type},
	JOURNAL = "Math. Ann.",
	FJOURNAL = "Mathematische Annalen",
	VOLUME = "281",
	YEAR = "1988",
	NUMBER = "1",
	PAGES = "123--133",
	ISSN = "0025-5831",
	CODEN = "MAANA",
	MRCLASS = "53C25 (14J29 32G07 53C55 58E11)",
	MRNUMBER = "944606 (89e:53075)",
	MRREVIEWER = "Akito Futaki",
	DOI = "10.1007/BF01449219",
	URL = "http://dx.doi.org/10.1007/BF01449219"
}

@Article{ GZ11,
	author = "Vincent Guedj and Ahmed Zeriahi",
	title = "{Stability of solutions to complex {Monge-Amp{\`e}re} equations in big cohomology classes}",
	journal = "to appear in MRL, arXiv:1112.1519",
	year = "2011"
}

@Book{ Laz1,
	author = "Robert Lazarsfeld",
	title = "{Positivity in Algebraic Geometry I}",
	publisher = "Springer",
	year = "2004"
}

@Book{ Laz2,
	author = "Robert Lazarsfeld",
	title = "{Positivity in Algebraic Geometry II}",
	publisher = "Springer",
	year = "2004"
}

@Article{ G12,
	title = {{K{\"a}hler-Einstein metrics with mixed Poincar{\'e} and cone singularities along a normal crossing divisor} },
	author = "Henri Guenancia",
	journal = "to appear in Ann. Inst. Fourier, arXiv:1201.0952",
	year = "2012"
}

@Article{ EGZ1,
	title = "A priori {$L^\infty$}-estimates for degenerate complex {M}onge-{A}mp{\`e}re equations",
	author = "Philippe Eyssidieux and Vincent Guedj and Ahmed Zeriahi",
	journal = "Int. Math. Res. Not. IMRN",
	pages = "Art. ID rnn 070, 8",
	year = "2008",
	issn = "1073-7928",
	url = "http://dx.doi.org/10.1093/imrn/rnn070",
	doi = "10.1093/imrn/rnn070",
	FJOURNAL = "International Mathematics Research Notices. IMRN",
	MRCLASS = "32W20 (32Q15)",
	MRNUMBER = "2439574 (2009f:32055)",
	MRREVIEWER = "S{\l}awomir Ko{\l}odziej"
}

@Article{ G2,
	title = {{Semi-stability of the tangent sheaf of singular varieties}},
	author = "Henri Guenancia",
	journal = "available on the author's web page",
	year = "2013"
}

@Article{ BDPP,
	title = {{The pseudo-effective cone of a compact K{\"a}hler manifold and varieties of negative Kodaira dimension}},
	author = "S. Boucksom and J.-P. Demailly and M. P{\u a}aun and T. Peternell",
	journal = "J. Alg. Geometry",
	year = "2013"
} 

@Article{ Wu,
	title = {K{\"a}hler-{E}instein metrics of negative {R}icci curvature on general quasi-projective manifolds},
	author = "Damin Wu",
	journal = "Comm. Anal. Geom.",
	pages = "395--435",
	volume = "16",
	number = "2",
	year = "2008",
	issn = "1019-8385",
	url = "http://projecteuclid.org/getRecord?id=euclid.cag/1216396331",
	FJOURNAL = "Communications in Analysis and Geometry",
	MRCLASS = "32Q20 (32W20 53C25)",
	MRNUMBER = "2425471 (2009g:32049)",
	MRREVIEWER = "Julien Keller"
}

@Article{ Wu2,
	title = {{Good K{\"a}hler metrics with prescribed singularities.}},
	author = "Damin Wu",
	journal = "Asian J. Math. ",
	pages = "131--150",
	volume = "13",
	number = "1",
	year = "2009"
}

@Article{ TY,
	title = {Complete {K}{\"a}hler manifolds with zero {R}icci curvature. {I}},
	author = "Gang Tian and Shing-Tung Yau",
	journal = "J. Amer. Math. Soc.",
	pages = "579--609",
	volume = "3",
	number = "3",
	year = "1990",
	issn = "0894-0347",
	url = "http://dx.doi.org/10.2307/1990928",
	doi = "10.2307/1990928",
	FJOURNAL = "Journal of the American Mathematical Society",
	MRCLASS = "53C55 (32C10 53C25)",
	MRNUMBER = "1040196 (91a:53096)",
	MRREVIEWER = "Alan Michael Nadel"
}

@Article{ Odaka,
	title = "{The GIT-stability of Polarised Varieties via discrepancy}",
	author = "Yuji Odaka",
	journal = "to appear in Ann. of Math, arXiv:0807.1716v5",
	year = "2008"
}


@article{DSong,
hyphenation = {american},
author = {Ved Datar and Jian Song},
title = {{A remark on Kähler metrics with conical singularities along a simple normal crossing divisor}},
date = {2013-09-19},
year = {2013},
eprinttype = {arxiv},
archivePrefix = {arXiv},
journal = {arXiv:1309.5013}
} 

@article{Yao,
hyphenation = {american},
author = {Chengjian Yao},
title = {Existence of Weak Conical K{\"a}hler-Einstein Metrics Along Smooth Hypersurfaces},
date = {2013-08-20},
year = {2013},
eprinttype = {arxiv},
archivePrefix = {arXiv},
journal = {arXiv:1308.4307}
} 

@Article{ koll,
	hyphenation = "american",
	author = "J{\'a}nos Koll{\'a}r",
	title = "Moduli of varieties of general type",
	date = "2010-08-03",
	year = "2010",
	eprinttype = "arxiv",
	archivePrefix = "arXiv",
	journal = "arXiv:1008.0621"
}

@Article{ Kob82,
	author = "Shoshichi Kobayashi",
	title = "{Curvature and stability of vector bundles.}",
	language = "English",
	journal = "Proc. Japan Acad., Ser. A",
	volume = "58",
	pages = "158--162",
	year = "1982"
}

@Article{ od,
	hyphenation = "american",
	author = "Yuji Odaka",
	title = "On the moduli of Kahler-Einstein Fano manifolds",
	date = "2013-02-08",
	year = "2013",
	eprinttype = "arxiv",
	archivePrefix = "arXiv",
	journal = "arXiv:1211.4833"
}

@Article{ Ber,
	hyphenation = "american",
	author = "Robert J. Berman",
	title = {{K-polystability of {$\Q$}-Fano varieties admitting K{\"a}hler-Einstein metrics}},
	date = "2012-10-21",
	year = "2012",
	eprinttype = "arxiv",
	archivePrefix = "arXiv",
	journal = "arXiv:1205.6214"
}

@Article{ CDS1,
	hyphenation = "american",
	author = "Xiuxiong Chen and Simon Donaldson and Song Sun",
	title = {{K{\"a}hler-Einstein metrics on Fano manifolds, I: approximation of metrics with cone singularities}},
	date = "2012-11-19",
	year = "2012",
	eprinttype = "arxiv",
	archivePrefix = "arXiv",
	journal = "arXiv:1211.4566"
}

@Article{ CDS2,
	hyphenation = "american",
	author = "Xiuxiong Chen and Simon Donaldson and Song Sun",
	title = {{K{\"a}hler-Einstein metrics on Fano manifolds, II: limits with cone angle less than {$2 \pi$}}},
	date = "2012-12-19",
	year = "2012",
	eprinttype = "arxiv",
	archivePrefix = "arXiv",
	journal = "arXiv:1212.4714"
}

@Article{ CDS3,
	hyphenation = "american",
	author = "Xiuxiong Chen and Simon Donaldson and Song Sun",
	title = {{K{\"a}hler-Einstein metrics on Fano manifolds, III: limits as cone angle approaches {$2\pi$} and completion of the main proof}},
	date = "2013-02-01",
	year = "2013",
	eprinttype = "arxiv",
	archivePrefix = "arXiv",
	journal = "arXiv:1302.0282"
}

@Article{ T,
	author = "Gang Tian",
	title = {{K-stability and K{\"a}hler-Einstein metrics}},
	year = "2013",
	journal = "arXiv:1211.4669",
	archivePrefix = "arXiv",
	date = "2013-01-28",
	eprinttype = "arxiv",
	hyphenation = "american"
}

@Article{ o-s,
	hyphenation = "american",
	author = "Yuji Odaka and Song Sun",
	title = "Testing log K-stability by blowing up formalism",
	date = "2011-12-06",
	year = "2011",
	eprinttype = "arxiv",
	archivePrefix = "arXiv",
	journal = "arXiv:1112.1353"
}

@Article{ od2,
	author = "Yuji Odaka",
	title = "{The Calabi conjecture and K-stability}",
	year = "2011",
	journal = "arXiv:1010.3597",
	archivePrefix = "arXiv",
	date = "2011-04-15",
	eprinttype = "arxiv",
	hyphenation = "american"
}

@Article{ Yau3,
	title = "A splitting theorem and an algebraic geometric characterization of locally {H}ermitian symmetric spaces",
	author = "Shing-Tung Yau",
	journal = "Comm. Anal. Geom.",
	pages = "473--486",
	volume = "1",
	number = "3-4",
	year = "1993",
	issn = "1019-8385",
	FJOURNAL = "Communications in Analysis and Geometry",
	MRCLASS = "32L07 (32C17 53C07 53C55)",
	MRNUMBER = "1266476 (95f:32039)",
	MRREVIEWER = "Kimio Miyajima"
}

@Book{ HL,
	author = "Daniel Huybrechts and Manfred Lehn",
	title = "{The geometry of moduli spaces of sheaves. 2nd ed.}",
	publisher = "{Cambridge: Cambridge University Press}",
	year = "2010"
}

@InCollection{ Enoki,
	AUTHOR = "Ichiro Enoki",
	TITLE = {Stability and negativity for tangent sheaves of minimal {K}{\"a}hler spaces},
	BOOKTITLE = "Geometry and analysis on manifolds ({K}atata/{K}yoto, 1987)",
	SERIES = "Lecture Notes in Math.",
	VOLUME = "1339",
	PAGES = "118--126",
	PUBLISHER = "Springer",
	ADDRESS = "Berlin",
	YEAR = "1988",
	MRCLASS = "32L10 (32J25 32L15)",
	MRNUMBER = "961477 (90a:32039)",
	MRREVIEWER = "Daniel Barlet",
	DOI = "10.1007/BFb0083051",
	URL = "http://dx.doi.org/10.1007/BFb0083051"
}

@Article{ RZ,
	title = "Continuity of extremal transitions and flops for {C}alabi-{Y}au manifolds",
	author = "Xiaochun Rong and Yuguang Zhang",
	journal = "J. Differential Geom.",
	pages = "233--269",
	volume = "89",
	number = "2",
	year = "2011",
	issn = "0022-040X",
	note = "Appendix B by Mark Gross",
	url = "http://projecteuclid.org/getRecord?id=euclid.jdg/1324477411",
	CODEN = "JDGEAS",
	FJOURNAL = "Journal of Differential Geometry",
	MRCLASS = "32Q25 (14J32 53C23)",
	MRNUMBER = "2863918",
	MRREVIEWER = "Michele Rossi"
}

@Article{ BHPS,
	hyphenation = "american",
	author = "Bhargav Bhatt and Wei Ho and Zsolt Patakfalvi and Christian Schnell",
	title = "Moduli of products of stable varieties",
	date = "2012-06-03",
	year = "2012",
	eprinttype = "arxiv",
	archivePrefix = "arXiv",
	journal = "arXiv:1206.0438"
}

@Article{ Mum77,
	author = "D. Mumford",
	title = "{Hirzebruch's proportionality theorem in the non-compact case.}",
	journal = "Invent. Math.",
	volume = "42",
	pages = "239--272",
	year = "1977",
	doi = "10.1007/BF01389790",
	classmath = "{14J25 (Special surfaces) 14M15 (Grassmannians, Schubert varieties, flag manifolds) 11F03 (Modular and automorphic functions) }"
}

@Article{ Alex,
	hyphenation = "american",
	author = "Valery Alexeev",
	title = "Log canonical singularities and complete moduli of stable pairs",
	date = "1996-08-17",
	year = "1996",
	eprinttype = "arxiv",
	journal = "arXiv:alg-geom/9608013"
}

@Article{ RZ2,
	title = "{Degenerations of Ricci-flat Calabi-Yau manifolds}",
	author = "Xiaochun Rong and Yuguang Zhang",
	journal = "arXiv:1206.3636",
	year = "2012"
}

@Article{ DS,
	title = {{Gromov-Hausdorff limits of K{\"a}hler manifolds and algebraic geometry}},
	author = "S. Donaldson and S. Sun",
	journal = "arXiv:1206.2609",
	year = "2012"
}

@InCollection{ KMM,
	AUTHOR = "Yujiro Kawamata and Katsumi Matsuda and Kenji Matsuki",
	TITLE = "Introduction to the minimal model problem",
	BOOKTITLE = "Algebraic geometry, {S}endai, 1985",
	SERIES = "Adv. Stud. Pure Math.",
	VOLUME = "10",
	PAGES = "283--360",
	PUBLISHER = "North-Holland",
	ADDRESS = "Amsterdam",
	YEAR = "1987",
	MRCLASS = "14E30 (14E05 14J40)",
	MRNUMBER = "946243 (89e:14015)",
	MRREVIEWER = "David R. Morrison"
}

@InCollection{ MY,
	AUTHOR = "Ngaiming Mok and Shing-Tung Yau",
	TITLE = {Completeness of the {K}{\"a}hler-{E}instein metric on bounded domains and the characterization of domains of holomorphy by curvature conditions},
	BOOKTITLE = "The mathematical heritage of {H}enri {P}oincar{\'e}, {P}art 1 ({B}loomington, {I}nd., 1980)",
	SERIES = "Proc. Sympos. Pure Math.",
	VOLUME = "39",
	PAGES = "41--59",
	PUBLISHER = "Amer. Math. Soc.",
	ADDRESS = "Providence, RI",
	YEAR = "1983",
	MRCLASS = "53C55 (32F15 53C25)",
	MRNUMBER = "720056 (85j:53068)",
	MRREVIEWER = "Hung-Hsi Wu"
}

@Article{ Langer,
	title = "Logarithmic orbifold Euler numbers of surfaces with applications",
	author = "Adrian Langer",
	journal = "arXiv:0012180",
	year = "2000",
	archivePrefix = "arXiv",
	date = "2000-12-19",
	eprinttype = "arxiv",
	hyphenation = "american"
}

@Article{ BB,
	author = "Robert J. Berman and Bo Berndtsson",
	title = {{Real Monge-Amp{\`e}re equations and K{\"a}hler-Ricci solitons on toric log Fano varieties}},
	year = "2012",
	archivePrefix = "arXiv",
	date = "2012-07-25",
	journal = "arXiv:1207.6128",
	eprinttype = "arxiv",
	hyphenation = "american"
}

@Article{ Fujino,
	title = "Finite generation of the log canonical ring in dimension four",
	author = "Osamu Fujino",
	journal = "Kyoto J. Math.",
	pages = "671--684",
	volume = "50",
	number = "4",
	year = "2010",
	issn = "2156-2261",
	url = "http://dx.doi.org/10.1215/0023608X-2010-010",
	doi = "10.1215/0023608X-2010-010",
	FJOURNAL = "Kyoto Journal of Mathematics",
	MRCLASS = "14E30 (14J35)",
	MRNUMBER = "2740690 (2012c:14032)",
	MRREVIEWER = "Thomas Eckl"
}

@Article{ Var,
	title = {{K{\"a}hler spaces and proper open morphisms.}},
	author = "Jean Varouchas",
	journal = "Math. Ann. ",
	pages = "13--52",
	volume = "283",
	number = "1",
	year = "1989",
	abstract = {{We prove among other results: (1) If $\pi$ : $X\to X'$ is a proper open surjective morphism of complex spaces with X K{\"a}hler and X' reduced then, if either $\pi$ is flat or X' normal, X' is K{\"a}hler. (2) If X is a K{\"a}hler space then the Barlet space $B\sb m(X)$ of compact complex m- cycles of X is K{\"a}hler. (3) Any reduced compact complex space in Fujiki's class ${\cal C}$ (holomorphic image of a compact K{\"a}hler space) is bimeromorphically equivalent to a compact K{\"a}hler manifold.}},
	doi = "10.1007/BF01457500",
	classmath = "{*53C55 (Complex differential geometry (global)) 32C15 (Complex spaces) }"
}

@Article{ FS,
	title = {{The moduli space of extremal compact K{\"a}hler manifolds and generalized Weil-Petersson metrics.}},
	author = "Akira Fujiki and Georg Schumacher",
	journal = "Publ. Res. Inst. Math. Sci. ",
	pages = "101--183",
	volume = "26",
	number = "1",
	year = "1990",
	doi = "10.2977/prims/1195171664",
	classmath = {{*32G13 (Analytic moduli problems) 53C55 (Complex differential geometry (global)) 32J27 (Compact K{\"a}hler manifolds) 32Q15 (K{\"a}hler manifolds) 14J15 (Analytic moduli, classification (surfaces)) }},
	reviewer = "{P.E.Newstead}"
}

@Book{ Kollar92,
	TITLE = "Flips and abundance for algebraic threefolds",
	AUTHOR = "J{\'a}nos Koll{\'a}r and others",
	NOTE = "Papers from the Second Summer Seminar on Algebraic Geometry held at the University of Utah, Salt Lake City, Utah, August 1991, Ast{\'e}risque No. 211 (1992)",
	PUBLISHER = "Soci{\'e}t{\'e} Math{\'e}matique de France",
	ADDRESS = "Paris",
	YEAR = "1992",
	PAGES = "1--258",
	ISSN = "0303-1179",
	MRCLASS = "14E30 (14E35 14M10)",
	MRNUMBER = "1225842 (94f:14013)",
	MRREVIEWER = "Mark Gross"
}

@Article{ Zar,
	title = "Analytical irreducibility of normal varieties",
	author = "Oscar Zariski",
	journal = "Ann. of Math. (2)",
	pages = "352--361",
	volume = "49",
	year = "1948",
	issn = "0003-486X",
	FJOURNAL = "Annals of Mathematics. Second Series",
	MRCLASS = "14.0X",
	MRNUMBER = "0024158 (9,460g)",
	MRREVIEWER = "I. S. Cohen"
}

@Article{ Li,
	title = "Greatest lower bounds on {R}icci curvature for toric {F}ano manifolds",
	author = "Chi Li",
	journal = "Adv. Math.",
	pages = "4921--4932",
	volume = "226",
	number = "6",
	year = "2011",
	issn = "0001-8708",
	url = "http://dx.doi.org/10.1016/j.aim.2010.12.023",
	doi = "10.1016/j.aim.2010.12.023",
	CODEN = "ADMTA4",
	FJOURNAL = "Advances in Mathematics",
	MRCLASS = "32Q20 (53C21 53C25)",
	MRNUMBER = "2775890 (2012h:32029)",
	MRREVIEWER = "Yanir A. Rubinstein"
}

@Article{ Sz,
	AUTHOR = "G{\'a}bor Sz{\'e}kelyhidi",
	TITLE = "Greatest lower bounds on the {R}icci curvature of {F}ano manifolds",
	JOURNAL = "Compos. Math.",
	FJOURNAL = "Compositio Mathematica",
	VOLUME = "147",
	YEAR = "2011",
	NUMBER = "1",
	PAGES = "319--331",
	ISSN = "0010-437X",
	MRCLASS = "32Q20 (53C25)",
	MRNUMBER = "2771134 (2011m:32037)",
	MRREVIEWER = "Ahmed Lesfari",
	DOI = "10.1112/S0010437X10004938",
	URL = "http://dx.doi.org/10.1112/S0010437X10004938"
}

@Article{ SW,
	title = "{The greatest Ricci lower bound, conical Einstein metrics and the Chern number inequality}",
	author = "Jian Song and Xiaowei Wang",
	journal = "arXiv:1207.4839",
	year = "2012",
	date = "2012-07-19",
	eprint = "1207.4839",
	eprinttype = "arxiv",
	hyphenation = "american"
}

@Article{ LS,
	title = {{Conical K{\"a}hler-Einstein metrics revisited}},
	author = "Chi Li and Song Sun",
	journal = "arXiv:1207.5011",
	year = "2012",
	date = "2012-10-07",
	eprint = "1207.5011",
	eprinttype = "arxiv",
	hyphenation = "american"
}

@Article{ Troyanov,
	title = "{Prescribing curvature on compact surfaces with conical singularities}",
	author = "Marc Troyanov",
	journal = "Trans. Am. Math. Soc.",
	pages = "793--821",
	volume = "324",
	number = "2",
	year = "1991",
	keywords = "{compact Riemann surface, conformal metric}, conical singularity, Gaussian curvature",
	doi = "10.2307/2001742",
	classmath = "{53C20 (Global Riemannian geometry, including pinching) 53A30 (Conformal differential geometry) }",
	reviewer = "{P.Eberlein (Chapel Hill)}"
}

@Article{ MO,
	title = "{Prescribed curvature and singularities of conformal metrics on Riemann surfaces}",
	author = "Robert C. McOwen",
	journal = "J. Math. Anal. Appl.",
	pages = "287--298",
	volume = "177",
	number = "1",
	year = "1993"
}

@Article{ Har2,
	title = "{Stable reflexive sheaves.}",
	author = "Robin Hartshorne",
	journal = "Math. Ann.",
	pages = "121--176",
	volume = "254",
	year = "1980",
	doi = "10.1007/BF01467074"
}

@Article{ Lub,
	title = "{Stability of Einstein-Hermitian vector bundles}",
	author = {Martin L{\"u}bke},
	journal = "Manuscr. Math.",
	pages = "245--257",
	volume = "42",
	year = "1983",
	language = "English"
}

@Article{ Lu,
	title = "{On holomorphic mappings of complex manifolds.}",
	author = "Y.-C. Lu",
	journal = "J. Diff. Geom.",
	pages = "299--312",
	volume = "2",
	year = "1968",
	keywords = "{complex functions}"
}

@Book{ Koba,
	title = "{Differential geometry of complex vector bundles.}",
	author = "Shoshichi Kobayashi",
	publisher = "{Princeton, NJ: Princeton University Press; Tokyo: Iwanami Shoten Publishers}",
	year = "1987",
	language = "English"
}

@Book{ JK1,
	title = "{Lectures on Resolution of Singularities.}",
	author = "J\'anos Koll\'ar",
	publisher = "{Annals of Mathematical Studies}",
	year = "2007",
	language = "English"
}

@InCollection{ Mum63,
	title = "Projective invariants of projective structures and applications",
	booktitle = "Proc. {I}nternat. {C}ongr. {M}athematicians ({S}tockholm, 1962)",
	author = "David Mumford",
	publisher = "Inst. Mittag-Leffler",
	address = "Djursholm",
	pages = "526--530",
	year = "1963",
	MRCLASS = "14.08 (14.20)",
	MRNUMBER = "0175899 (31 \#175)"
}

@Article{ Take,
	title = "{Stable vector bundles on algebraic surfaces.}",
	author = "Fumio Takemoto",
	journal = "Nagoya Math. J.",
	pages = "29--48",
	volume = "47",
	year = "1972",
	classmath = "{14F05 (Sheaves, derived categories of sheaves, etc.) 14J10 (Families, algebraic moduli, classification (surfaces)) }"
}

@Article{ NS,
	title = "Stable and unitary vector bundles on a compact {R}iemann surface",
	author = "M. S. Narasimhan and C. S. Seshadri",
	journal = "Ann. of Math. (2)",
	pages = "540--567",
	volume = "82",
	year = "1965",
	issn = "0003-486X",
	FJOURNAL = "Annals of Mathematics. Second Series",
	MRCLASS = "14.10",
	MRNUMBER = "0184252 (32 \#1725)",
	MRREVIEWER = "M. F. Atiyah"
}

@Article{ Don83,
	title = "A new proof of a theorem of {N}arasimhan and {S}eshadri",
	author = "S. K. Donaldson",
	journal = "J. Differential Geom.",
	pages = "269--277",
	volume = "18",
	number = "2",
	year = "1983",
	issn = "0022-040X",
	url = "http://projecteuclid.org/getRecord?id=euclid.jdg/1214437664",
	CODEN = "JDGEAS",
	FJOURNAL = "Journal of Differential Geometry",
	MRCLASS = "32L05 (14F05 53C05)",
	MRNUMBER = "710055 (85a:32036)",
	MRREVIEWER = "P. E. Newstead"
}

@Article{ Don87,
	title = "Infinite determinants, stable bundles and curvature",
	author = "S. K. Donaldson",
	journal = "Duke Math. J.",
	pages = "231--247",
	volume = "54",
	number = "1",
	year = "1987",
	issn = "0012-7094",
	url = "http://dx.doi.org/10.1215/S0012-7094-87-05414-7",
	doi = "10.1215/S0012-7094-87-05414-7",
	CODEN = "DUMJAO",
	FJOURNAL = "Duke Mathematical Journal",
	MRCLASS = "32L15 (53C25 58E15)",
	MRNUMBER = "885784 (88g:32046)",
	MRREVIEWER = "Daniel S. Freed"
}

@Article{ UY2,
	title = "A note on our previous paper: ``{O}n the existence of {H}ermitian-{Y}ang-{M}ills connections in stable vector bundles'' [{C}omm.\ {P}ure {A}ppl.\ {M}ath.\ {\bf 39} (1986), {S}257--{S}293; {MR}0861491 (88i:58154)]",
	author = "K. Uhlenbeck and S.-T. Yau",
	journal = "Comm. Pure Appl. Math.",
	pages = "703--707",
	volume = "42",
	number = "5",
	year = "1989",
	issn = "0010-3640",
	url = "http://dx.doi.org/10.1002/cpa.3160420505",
	doi = "10.1002/cpa.3160420505",
	CODEN = "CPAMA",
	FJOURNAL = "Communications on Pure and Applied Mathematics",
	MRCLASS = "58E15 (32L15 53C05 58G05)",
	MRNUMBER = "997570 (90i:58029)"
}

@Article{ GKK,
	title = "Extension theorems for differential forms and {B}ogomolov-{S}ommese vanishing on log canonical varieties",
	author = "Daniel Greb and Stefan Kebekus and S{\'a}ndor J. Kov{\'a}cs",
	journal = "Compos. Math.",
	pages = "193--219",
	volume = "146",
	number = "1",
	year = "2010",
	issn = "0010-437X",
	url = "http://dx.doi.org/10.1112/S0010437X09004321",
	doi = "10.1112/S0010437X09004321",
	FJOURNAL = "Compositio Mathematica",
	MRCLASS = "14F17 (14J17)",
	MRNUMBER = "2581247 (2011c:14054)",
	MRREVIEWER = "Mirroslav Tzanov Yotov"
}

\end{document}